\documentclass[reqno]{amsart}

\usepackage{amssymb,amsfonts,amsmath,mathdots,amsthm}
\usepackage{float}

\usepackage{tikz-cd}
\usepackage{lmodern}
\usepackage{mathrsfs}
\usepackage[sc]{mathpazo}
\DeclareMathAlphabet{\mathfrak}{U}{euf}{m}{n}
\usepackage{mathtools}  
\usepackage{calligra}
\usepackage{xcolor}
\usepackage[hidelinks]{hyperref}
\usepackage{accents}
\usepackage{subcaption}
\usepackage{spectralsequences}
\usepackage{dashbox}
\usepackage{enumitem}
\usepackage{tabu}

\usepackage{hyperref}

\hypersetup{colorlinks=true,
	linkcolor=magenta,
	citecolor=magenta}

\setlist[itemize]{leftmargin=*}

\newcommand\zero{0}

\usepackage{relsize}
\usepackage{etoolbox,fancyhdr}
\usepackage{indentfirst}  

\usetikzlibrary{arrows,arrows.meta,shapes.misc,decorations.markings,positioning,patterns}

\let\oldproofname=\proofname
\renewcommand{\proofname}{\bfseries\textup{\oldproofname}}

\mathtoolsset{showonlyrefs} 

\DeclarePairedDelimiter\ev{\langle}{\rangle}

\DeclareMathOperator{\Z}{\mathbb{Z}}
\DeclareMathOperator{\F}{\mathbb{F}}
\DeclareMathOperator{\R}{\mathbb{R}}

\DeclareMathOperator{\Tr}{Tr}
\DeclareMathOperator{\Res}{Res}

\theoremstyle{plain}

\newtheorem{prop}{Proposition}[section]
\newtheorem{lem}[prop]{Lemma}

\newtheorem{cor}[prop]{Corollary}

\theoremstyle{definition}
\newtheorem{rem}[prop]{Remark}
\makeatother
\title{The $RO(C_4)$ cohomology of the infinite real projective space}
\author{Nick Georgakopoulos}

\begin{document}

	\begin{abstract}Following the Hu-Kriz method of computing the $C_2$ genuine dual Steenrod algebra $(H\F_2)_{\bigstar}(H\F_2)$, we calculate the  $C_4$ equivariant Bredon cohomology of the classifying space $\R P^{\infty \rho}=B_{C_4}\Sigma_{2}$ as an $RO(C_4)$ graded Green-functor. We prove that as a module over the homology of a point (which we also compute), this cohomology is not flat. As a result, it can't be used as a test module for obtaining generators in  $(H\F_2)_{\bigstar}(H\F_2)$ as Hu-Kriz do in the $C_2$ case.
	\end{abstract}
	
	\maketitle{}

	\setcounter{tocdepth}{4}
	\tableofcontents
	\section{Introduction}\label{Intro}

Historically, computations in stable equivariant homotopy theory have been much more difficult than their nonequivariant counterparts, even when the groups involved are as simple as possible (i.e. cyclic). In recent years, there has been a resurgence in such calculations for power $2$-cyclic groups $C_{2^n}$, owing to the crucial involvement of $C_8$-equivariant homology in the solution of the Kervaire invariant problem \cite{HHR16}. 

The case of $G=C_2$ is the simplest and most studied one, partially due to its connections to motivic homotopy theory over $\R$ by means of realization functors \cite{HO14}. It all starts with the $RO(C_2)$ homology of a point, which was initially described in \cite{Lew88}. The types of modules over it that can arise as the equivariant homology of spaces were described in \cite{CMay18}, and this description was subsequently used in the computation of the $RO(C_2)$ homology of $C_2$-surfaces in \cite{Haz19}. The $C_2$-equivariant dual Steenrod algebra (in characteristic $2$) was computed in \cite{HK96} and gives rise to a $C_2$-equivariant Adams spectral sequence that has been more recently leveraged in \cite{IWX20}. Another application of the Hu-Kriz computation is the definition of equivariant Dyer-Lashof operations by \cite{Wil19} in the $\F_2$-homology of $C_2$-spectra with symmetric multiplication. Many of these results rely on the homology of certain spaces being free as modules over the homology of a point, and there is a robust theory of such free spectra described in \cite{Hil19}.
	
The case of $G=C_4$ has been much less explored and is indeed considerably more complicated. This can already be seen in the homology of a point in integer coefficients (see \cite{Zeng} and \cite{Geo19}) and the case of $\F_2$ coefficients is not much better (compare subsections \ref{C2Tate} and \ref{ROC4HomolPt} for the $C_2$ and $C_4$ cases respectively). The greater complexity in the ground ring (or to be more precise, ground Green functor), means that modules over it can also be more complicated and indeed, certain freeness results that are easy to obtain in the $C_2$ case no longer hold when generalized to $C_4$ (compare subsection \ref{BC2S2} with sections \ref{BC4S2dec} and \ref{BC4S2ss}).

The computation of the dual Steenrod algebra relies on the construction of Milnor generators. Nonequivariantly, the Milnor generators $\xi_i$ of the mod $2$ dual Steenrod algebra can be defined through the completed coaction of the dual Steenrod algebra on the cohomology of $B\Sigma_2=\R P^{\infty}$: $H^*(BC_{2+};\F_2)=\F_2[x]$ and the completed coaction $\F_2[x]\to (H\F_2)_*(H\F_2)[[x]]$ is: 
\begin{equation}
x\mapsto \sum_ix^{2^i}\otimes \xi_i
\end{equation} 
In the $C_2$-equivariant case, the space replacing $B\Sigma_2$ is the equivariant classifying space $B_{C_2}\Sigma_2$. This is still $\R P^{\infty}$ but now equipped with a nontrivial $C_2$ action (described in subsection \ref{BC2S2}). Over the homology of a point, we no longer have a polynomial algebra on a single generator $x$, but rather a polynomial algebra on two generators $c,b$ modulo the relation
\begin{equation}
c^2=a_{\sigma}c+u_{\sigma}b
\end{equation}
where $a_{\sigma},u_{\sigma}$ are the $C_2$-Euler and orientation classes respectively (defined in section \ref{Notations}). As a module, this is still free over the homology of a point, and the completed coaction is
\begin{gather}
c\mapsto c\otimes 1+\sum_ib^{2^i}\otimes \tau_i\\
b\mapsto \sum_ib^{2^i}\otimes \xi_i
\end{gather}
The $\tau_i,\xi_i$ are the $C_2$-equivariant analogues of the Milnor generators, and Hu-Kriz show that they span the genuine dual Steenrod algebra.

For $C_4$, the cohomology of $B_{C_4}\Sigma_2$ is significantly more complicated (see section \ref{BC4S2ss}) and most importantly is \emph{not} a free module over the homology of a point. In fact, it's not even flat (Proposition \ref{NonFlatnessIs}) bringing into question whether we even have a coaction by the dual Steenrod algebra in this case.

There is another related reason to consider the space $B_{C_4}\Sigma_2$. In \cite{Wil19}, the author describes a framework for equivariant total power operations over an $H\F_2$ module $A$ equipped with a symmetric multiplication. The total power operation is induced from a map of spectra
\begin{equation}
A\to A^{t\Sigma_{[2]}} 
\end{equation}
where $(-)^{t\Sigma_{[2]}}$ is a variant Tate construction defined in \cite{Wil19}. 

In the nonequivariant case, $A\to A^{t\Sigma_{[2]}}$ induces a map $A_*\to A_*((x))$ and the Dyer-Lashof operations $Q^i$ can be obtained as the components of this map:
\begin{equation}
Q(x)=\sum_iQ^i(x)x^i
\end{equation}

In the $C_2$ equivariant case, we have a map $A_{\bigstar}\to A_{\bigstar}[c,b^{\pm}]/(c^2=a_{\sigma}c+u_{\sigma}b)$ and we get power operations
\begin{equation}
Q(x)=\sum_iQ^{i\rho}(x)b^i+\sum_iQ^{i\rho+\sigma}(x)cb^i
\end{equation}
When $A=H\F_2$, $A_{\bigstar}[c,b^{\pm}]/(c^2=a_{\sigma}c+u_{\sigma}b)$ is the cohomology of $B_{C_2}\Sigma_2$ localized at the class $b$.

For $C_4$ we would have to use the cohomology of $B_{C_4}\Sigma_2$ (localized at a certain class) but that is no longer free, meaning that the resulting power operations would have extra relations between them and further complicating the other arguments in \cite{Wil19}.

The computation of $H^{\bigstar}(B_{C_4}\Sigma_{2+};\F_2)$ also serves for a test case of $RO(G)$ homology computations for equivariant classifying spaces where $G$ is not of prime order. We refer the reader to \cite{Shu14}, \cite{Cho18}, \cite{Wil19}, \cite{SW21} for such computations in the $G=C_p$ case.

As for the organization of this paper, section \ref{Notations} describes the conventions and notations that we shall be using throughout this document, as well as the Tate diagram for a group $G$ and a $G$-equivariant spectrum.

Subsections \ref{C2Tate} and \ref{C4Tate} describe the Tate diagram for $C_2$ and $C_4$ respectively using coefficients in the constant Mackey functor $\F_2$.

In section \ref{Classifying} we define equivariant classifying spaces $B_GH$ and briefly explain the elementary computation of the cohomology of $B_{C_2}\Sigma_2$ (this argument also appears in \cite{Wil19}).

In section \ref{BC4Sigma2Summary} we present the result of the computation of $H^{\bigstar}(B_{C_4}\Sigma_{2+};\F_2)$ and prove that it's not flat as a Mackey functor module over $(H\F_2)_{\bigstar}$. Sections \ref{BC4S2dec} and \ref{BC4S2ss} contain the proofs of the computation of the cohomology of $B_{C_4}\Sigma_2$.

We have included two appendices in the end; Appendix \ref{AppendixSS} contains pictures of the spectral sequence converging to $H^{\bigstar}(B_{C_4}\Sigma_{2+};\F_{2})$ while Appendix \ref{AppendixPoint} contains a detailed description of $H^{\bigstar}(S^0;\F_{2})$, which is the ground Green functor over which all our Mackey functors are modules.

To aid in the creation of these appendices, we extensively used the computer program of \cite{Geo19} available \href{https://github.com/NickG-Math/Mackey}{here}. In fact, we have introduced new functionality in the software that computes the $RO(G)$-graded homology of spaces such as $B_{C_4}\Sigma_2$ given an explicit equivariant CW decomposition (such as we discuss in subsection \ref{CWDecompOrbits}). This assisted in the discovery of a nontrivial $d^2$ differential in the spectral sequence of $B_{C_4}\Sigma_2$ (see Remark \ref{ProgramDifferential}), although the provided proof is independent of the computer computation.
	
\subsection*{Acknowledgment}We would like to thank Dylan Wilson for answering our questions regarding his paper \cite{Wil19} as well as \cite{HK96}. We would also like to thank Peter May for his numerous editing suggestions, that vastly improved the readability of this paper.

	\section{Conventions and notations}\label{Notations}
	
We will use the letter $k$ to denote  the field $\F_2$, the constant Mackey functor $k=\underline{\F_2}$ and the corresponding Eilenberg-MacLane spectrum $Hk$. The meaning should always be clear from the context.

All our homology and cohomology will be in $k$ coefficients. \medbreak

The data of a $C_4$ Mackey functor $M$ can be represented by a diagram displaying the values of $M$ on orbits, its restriction and transfer maps and the actions of the Weyl groups. We shall refer to $M(C_4/C_4), M(C_4/C_2), M(C_4/e)$ as the top, middle and bottom levels of the Mackey functor $M$ respectively. The Mackey functor diagram takes the form:

 \begin{equation}
M=\begin{tikzcd}
M(C_4/C_4)\ar[d, "\Res^4_2" left, bend right]\\
M(C_4/C_2)\ar[u, "\Tr_2^4" right,bend right]\ar[d, "\Res^2_1" left, bend right]\ar[loop right, "C_4/C_2" right]\\
M(C_4/e)\ar[u, "\Tr_1^2" right,bend right]\ar[loop right, "C_4" right]
\end{tikzcd}
\end{equation}

If $X$ is a $G$-spectrum then $X_{\bigstar}$ denotes the $RO(G)$-graded $G$-Mackey functor defined on orbits as 
\begin{equation}
	X_{\bigstar}(G/H)=X_{\bigstar}^H=\pi^H(S^{-\bigstar}\wedge X)=[S^{\bigstar},X]^H
\end{equation}
The index $\bigstar$ will always be an element of the real representation ring $RO(G)$. \medbreak
	
$RO(C_4)$ is spanned by the irreducible representations $1,\sigma, \lambda$ where $\sigma$ is the $1$-dimensional sign representation and $\lambda$ is the $2$-dimensional representation given by rotation by $\pi/2$.

For $V=\sigma$ or $V=\lambda$, denote by $a_V\in k_{-V}^{C_4}$ the Euler class induced by the inclusion of north and south poles $S^0\hookrightarrow S^V$; also denote by $u_V\in k_{|V|-V}^{C_4}$ the orientation class generating the Mackey functor  $k_{|V|-V}=k$.

We will use the notation $\bar a_V,\bar u_V$ to denote the restrictions of $a_V,u_V$ to middle level, and $\overline{ \bar u}_V$ to denote the restriction of $u_V$ to bottom level.

We also write $a_{\sigma_2}\in k_{-\sigma_2}^{C_2}$ and $u_{\sigma_2}\in k_{1-\sigma_2}^{C_2}$ for the $C_2$ Euler and orientation classes, where $\sigma_2$ is the sign representation of $C_2$.

The Gold Relation (\cite{HHR16}) takes the following form in $k$ coefficients:
			\begin{equation}
	a_{\sigma}^2u_{\lambda}=0
	\end{equation}

	Let $EG$ be a contractible free $G$-space and $\tilde EG$ be the cofiber of the collapse map $EG_+\to S^0$.	We use the notation $X_h=EG_+\wedge X$, $\tilde X=\tilde EG\wedge X$, $X^h=F(EG_+,X)$ and $X^t=\widetilde{X^h}$.
		
	The Tate diagram (\cite{GM95}) then takes the form:
	\begin{center}\begin{tikzcd}
		X_h\ar[d,"\simeq"]\ar[r]&X\ar[r]\ar[d]& \tilde X\ar[d]\\
		X_h\ar[r]& X^h\ar[r]& X^t
		\end{tikzcd}\end{center}	
	The square on the right is a homotopy pullback diagram and is called the Tate square. 
	
	Applying $\pi_{\bigstar}^G$ on the Tate diagram gives
	\begin{center}\begin{tikzcd}
		X_{hG\bigstar}\ar[r]\ar[d,"\simeq"]&X_{\bigstar}^G\ar[r]\ar[d]& \tilde X_{\bigstar}^G\ar[d]\\
		X_{hG\bigstar}\ar[r]&X^{hG}_{\bigstar}\ar[r]& X^{tG}_{\bigstar}
		\end{tikzcd}\end{center}

	\section{\texorpdfstring{The Tate diagram for $C_2$ and $C_4$}{The Tate diagram for C2 and C4}}\label{TateC2C4}
	
	\subsection{\texorpdfstring{The Tate diagram for $C_2$}{The Tate diagram for C2}}\label{C2Tate}

	For $X=k$ and $G=C_2$ the corners of the Tate square are:
	\begin{gather}
	k_{\bigstar}^{C_2}=k[a_{\sigma_2},u_{\sigma_2}]\oplus k\{\frac{\theta_{\sigma_2}}{a_{\sigma_2}^iu_{\sigma_2}^j}\}_{i,j\ge 0}\\
	\tilde k_{\bigstar}^{C_2}=k[a_{\sigma_2}^{\pm},u_{\sigma_2}]\\
	k_{\bigstar}^{hC_2}=k[a_{\sigma_2},u_{\sigma_2}^{\pm}]\\
	k_{\bigstar}^{tC_2}=k[a_{\sigma_2}^{\pm},u_{\sigma_2}^{\pm}]
	\end{gather}
	where $\theta_{\sigma_2}=\Tr_1^2(\Res^2_1(u_{\sigma_2})^{-2})$. The map $k_h\to k$ in the Tate diagram induces
	\begin{gather}
	k_{hC_2\bigstar}=\Sigma^{-1}k^{tC_2}_{\bigstar}/k^{hC_2}_{\bigstar}\to k^{C_2}_{\bigstar}\\
	a_{\sigma_2}^{-i}u_{\sigma_2}^{-j}\mapsto \frac{\theta_{\sigma_2}}{a_{\sigma_2}^iu_{\sigma_2}^{j-1}}
	\end{gather}

\subsection{\texorpdfstring{The $RO(C_4)$ homology of a point}{The RO(C4) homology of a point}}\label{ROC4HomolPt} The $RO(C_4)$ homology of a point (in $k$ coefficients) is significantly more complicated than the $RO(C_2)$ one (see \cite{Geo19} for the integer coefficient case). Appendix \ref{AppendixPoint} contains a very detailed description of it, and the goal in this subsection is to provide a more compact version.

The top level is:
\begin{equation}\label{C4homologypoint}
	k_{\bigstar}^{C_4}\doteq k[a_{\sigma},u_{\sigma},a_{\lambda},u_{\lambda},\frac{u_{\lambda}}{u_{\sigma}^i}, \frac{a_{\sigma}^2}{a_{\lambda}^i}]\oplus k[a_{\lambda}^{\pm}]\{\frac{\theta}{a_{\sigma}^iu_{\sigma}^j}\} \oplus 
	k\{\frac{\frac{\theta}{a_{\lambda}}a_{\sigma}^{1+\epsilon}}{u_{\sigma}^ia_{\lambda}^j}\}\oplus k[u_{\sigma}^{\pm}]\{\frac{\frac{\theta}{a_{\lambda}}a_{\sigma}^{1+\epsilon} }{a_{\lambda}^ju_{\lambda}^{1+m}} \}
\end{equation}
where the indices $i,j,m$ range in $0,1,2,...$ and $\epsilon$ ranges in $0, 1$.

The use of $\doteq$ as opposed to $=$ is meant to signify some subtlety present in \eqref{C4homologypoint} that needs to be clarified before the equality can be used. This subtlety has to do with how quotients are defined (cf \cite{Geo19}) and how elements multiply (the multiplicative relations). We begin this process of interpreting \eqref{C4homologypoint} with the definition of $\theta$: $\theta=\Tr_2^4(\bar u_{\sigma}^{-2})$. We further introduce the elements
\begin{gather}
	x_{n,m}=\Tr_2^4(\overline{\bar u}_{\sigma}^{-n}\overline{\bar u}_{\lambda}^{-m})=\frac{x_{0,1}}{u_{\sigma}^nu_{\lambda}^{m-1}}\text{ , }m\ge 1
\end{gather}
where
\begin{equation}
	x_{0,1}=a_{\sigma}^2\frac{\theta}{a_{\lambda}}=\theta\frac{a_{\sigma}^2}{a_{\lambda}}
\end{equation}
With this notation, the second curly bracket in \eqref{C4homologypoint} contains elements of the form
\begin{equation}
\frac{x_{n,1}}{a_\lambda^i}\text{ , }\frac{x_{n,1}}{a_\sigma a_\lambda^i}
\end{equation}
and the third contains
\begin{equation}
\frac{x_{n,m}}{a_\lambda^i}\text{ , }\frac{x_{n,m}}{a_\sigma a_\lambda^i}\text{ , }m>1
\end{equation}
The behavior of the $x_{n,m}$ depends crucially on whether $m=1$ or not: $x_{n,1}u_{\sigma}=0$ but $x_{n,m}u_{\sigma}\neq 0$ for $m>1$; the $x_{n,1}$ are infinitely $a_{\sigma}$ divisible since:
\begin{equation}
\frac{x_{n,1}}{a_\sigma^2}=\frac{\theta}{u_\sigma^na_\lambda}
\end{equation}
 while the $x_{n,m}$, $m>1$, can only be divided by $a_{\sigma}$ once. That's why we separate them into two distinct summands in \eqref{C4homologypoint}.

The third curly bracket in \eqref{C4homologypoint} for $\epsilon=0$ consists of quotients of
\begin{equation}
	s:=\frac{\frac{\theta}{a_{\lambda}}a_{\sigma}}{u_{\lambda}}u_{\sigma}=\frac{x_{0,2}u_\sigma}{a_\sigma}
\end{equation}
which is the mod $2$ reduction of the element $s$ from \cite{Geo19}. Note that $su_\lambda=sa_\lambda=0$.

The quotients in the RHS of \eqref{C4homologypoint} are all chosen coherently (cf \cite{Geo19}), that is we always have the cancellation property:
\begin{equation}
	z\cdot \frac{y}{xz}=\frac{y}{x}
\end{equation}
We also have that
\begin{equation}
	\frac{x}{y}\cdot \frac{z}{w}=\frac{xz}{yw}
\end{equation}
as long as $xz\neq 0$ (this condition is necessary: $(\theta/a_{\lambda})a_{\sigma}\neq 0$ is not $(\theta a_{\sigma})/a_{\lambda}$ as $\theta a_{\sigma}=0$).

To compute any product of two elements in the RHS of \eqref{C4homologypoint} we follow the following procedure: 
\begin{itemize}
	\item If both elements involve $\theta$ then the product is automatically $0$.
	\item If neither element involves $\theta$ then perform all possible cancellations and use the relation
	\begin{equation}
		\frac{u_{\lambda}}{u_{\sigma}^i}\cdot \frac{a_{\sigma}^2}{a_{\lambda}^j}=0
	\end{equation}
	\item If only one element involves $\theta$ perform all possible cancellations and use 
	\begin{equation}
		\frac{x}{y}\cdot \frac{z}{w}=\frac{xz}{yw}
	\end{equation}
	as long as $xz$ appears in \eqref{C4homologypoint}. If the resulting element appears in \eqref{C4homologypoint} then that's the product; if not then the product is $0$.
\end{itemize} 
These are all the remarks needed to properly interpret the formula in \eqref{C4homologypoint}  for the top level $k^{C_4}_{\bigstar}$.\medbreak

The middle level is:
\begin{equation}\label{C4homologyofapointmidlevel}
	k^{C_2}_{\bigstar}\doteq k[\bar a_{\lambda},\bar u_{\lambda},\sqrt{\bar a_{\lambda}\bar u_{\lambda}},\bar u_{\sigma}^{\pm}] \oplus k[\bar u_{\sigma}^{\pm}]\{\frac{v}{\bar a_{\lambda}^i\bar u_{\lambda}^j \sqrt{\bar a_{\lambda}\bar u_{\lambda}}^{\epsilon}}\}
\end{equation}
Here, $\sqrt{\bar a_{\lambda}\bar u_{\lambda}}$ is the (unique) element whose square is $\bar a_{\lambda}\bar u_{\lambda}$ and $v$ is defined by $v=\Tr_1^2(\overline{\bar u}_{\lambda}^{-1})$. Furthermore,
\begin{gather}
	\Tr_2^4(\sqrt{\bar a_{\lambda}\bar u_{\lambda}})=\frac{a_{\sigma}u_{\lambda}}{u_{\sigma}}\\
	\Tr_2^4(v)=x_{0,1}\\
	\bar s:= \Res^4_2(s)=\frac{v}{\sqrt{\bar a_{\lambda}\bar u_{\lambda}}}\\
	\Res^4_2\Big(\frac{a_{\sigma}^2}{a_{\lambda}}\Big)=v\bar u_{\sigma}^2
\end{gather}
The interpretation of \eqref{C4homologyofapointmidlevel} is complete. In terms of the notation of the $C_2$ generators,
\begin{gather}
	\bar a_{\lambda}=a_{\sigma_2}^2\text{ , }\bar u_{\lambda}=u_{\sigma_2}^2\text{ , }\sqrt{\bar a_{\lambda}\bar u_{\lambda}}=a_{\sigma_2}u_{\sigma_2}\text{ , }v=\theta_{\sigma_2}
\end{gather}
Finally the bottom level is very simple:
\begin{equation}
	k^e_{\bigstar}=k[\overline{\bar u}_{\lambda}^{\pm},\overline{\bar u}_{\sigma}^{\pm}]
\end{equation}
For more details, consult the Appendix \ref{AppendixPoint}.\medbreak

	\subsection{\texorpdfstring{The Tate diagram for $C_4$}{The Tate diagram for C4}}\label{C4Tate}
	 
Using the notation of the previous subsection, the corners of the Tate square are:
	 \begin{gather}
	 k_{\bigstar}^{C_4}\doteq k[a_{\sigma},u_{\sigma},a_{\lambda},u_{\lambda},\frac{u_{\lambda}}{u_{\sigma}^i}, \frac{a_{\sigma}^2}{a_{\lambda}^i}]\oplus k[a_{\lambda}^{\pm}]\{\frac{\theta}{a_{\sigma}^iu_{\sigma}^j}\} \oplus 
	 	k\{\frac{\frac{\theta}{a_{\lambda}}a_{\sigma}^{1+\epsilon}}{u_{\sigma}^ia_{\lambda}^j}\}\oplus k[u_{\sigma}^{\pm}]\{\frac{\frac{\theta}{a_{\lambda}}a_{\sigma}^{1+\epsilon} }{a_{\lambda}^ju_{\lambda}^{1+m}} \}\\
	 \tilde k^{C_4}_{\bigstar}=a_{\lambda}^{-1}k^{C_4}_{\bigstar}\doteq k[a_{\sigma},u_{\sigma},a_{\lambda}^{\pm},u_{\lambda},\frac{u_{\lambda}}{u_{\sigma}^i}]\oplus k[a_{\lambda}^{\pm}]\{\frac{\theta }{a_{\sigma}^iu_{\sigma}^j}\}\\
	 k^{hC_4}_{\bigstar}=k[a_{\sigma},u_{\sigma}^{\pm},a_{\lambda},u_{\lambda}^{\pm}]/a_{\sigma}^2\\
	 k^{tC_4}_{\bigstar}=k[a_{\sigma},u_{\sigma}^{\pm},a_{\lambda}^{\pm},u_{\lambda}^{\pm}]/a_{\sigma}^2
	 \end{gather}
	 The map $k_h\to k$ in the Tate diagram induces
	 \begin{gather} k_{hC_4\bigstar}=\Sigma ^{-1}k^{tC_4}_{\bigstar}/k^{hC_4}_{\bigstar}\to k^{C_4}_{\bigstar}\\
u_{\sigma}^{-i}a_{\lambda}^{-j}u_{\lambda}^{-m}\mapsto \frac{\frac{\theta}{a_{\lambda}}a_{\sigma}}{u_{\sigma}^{i-1}a_{\lambda}^{j-1}u_{\lambda}^m}
	 \end{gather}

\section{Equivariant classifying spaces}\label{Classifying}
	 
$E_GK$ is a $G\times K$ space that's $K$-free and $(E_GK)^{\Gamma}$ is contractible for any subgroup $\Gamma\subseteq G\times K$ with $\Gamma\cap (\{1\}\times K)=\{1\}$ (a graph subgroup). The spaces $E_G\Sigma_n$ are those appearing in a $G$-$E_{\infty}$-operad. 

We define the equivariant classifying space
\begin{equation}
B_GK=E_GK/K
\end{equation}

\subsection{\texorpdfstring{The case of $C_2$}{The case of C2}}\label{BC2S2} For $G=C_2$, the spaces $B_{C_2}\Sigma_2$ are used in the computation of the $C_2$ dual Steenrod algebra by Hu-Kriz (\cite{HK96}) and for the construction of the total $C_2$-Dyer-Lashof operations in \cite{Wil19}. Both use the computation
\begin{equation}
k_{C_2}^{\bigstar}(B_{C_2}\Sigma_{2+})=k_{C_2}^{\bigstar}[c,b]/(c^2=a_{\sigma_2}c+u_{\sigma_2}b)
\end{equation}
where $c,b$ are classes in cohomological degrees $\sigma_2, 1+\sigma_2$ respectively. Let us note here that $B_{C_2}\Sigma_2$ is $\R P^{\infty}$ with a nontrivial $C_2$ action; the restrictions of $c,b$ are the generators of degree $1,2$ of $k_*(\R P^{\infty})$.

We shall now summarize this computation, since part of it will be needed for the analogous computation when $G=C_4$ which takes place in sections \ref{BC4Sigma2Summary}-\ref{BC4S2ss}.

Let $\sigma,\tau$ be the sign representations of $C_2,\Sigma_2$ respectively and $\rho=1+\sigma$. Then $E_{C_2}\Sigma_2=S(\infty(\rho\otimes \tau))$; the graph subgroups of $C_2\times \Sigma_2$ are $C_2, \Delta$ and their orbits correspond to the cells
\begin{gather}
\frac{C_2\times \Sigma_2}{C_2}=S(1\otimes \tau)\\
\frac{C_2\times \Sigma_2}{\Delta}=S(\sigma\otimes \tau)
\end{gather}
\cite{Wil19} defines a filtration on $E_{C_2}\Sigma_2$ given by
\begin{equation}
S(1\otimes \tau)\subseteq S(\rho\otimes \tau)\subseteq S((\rho+1)\otimes \tau)\subseteq S(2\rho\otimes \tau) \subseteq\cdots
\end{equation}
and whose quotients (after adjoining disjoint basepoints) are 
\begin{gather}
gr_{2j+1}=\frac{C_2\times \Sigma_2}{\Delta}_+\wedge S^{(j+1)\rho\otimes 1-1}\\
gr_{2j}=\Sigma_{2+}\wedge S^{j\rho\otimes 1}
\end{gather}
Quotiening $\Sigma_2$ gives a filtration for $B_{C_2}\Sigma_{2+}$ with
\begin{gather}
gr_{2j+1}=S^{(j+1)\rho\otimes 1-1}\\
gr_{2j}=S^{j\rho\otimes 1}
\end{gather}
Applying $k^{\bigstar}$ yields a spectral sequence
\begin{equation}
E^1=k^{\bigstar}\{e^{j\rho},e^{j\rho+\sigma}\}\implies k^{\bigstar}(B_{C_2}\Sigma_{2+})
\end{equation}
	 of modules over the Green functor $k^{\bigstar}$. The fact that the differentials are module maps gives $E_1=E_2$ for degree reasons. Furthermore, the vanishing of the $RO(C_2)$ homology of a point in a certain range gives $E_2=E_{\infty}$. The $E_{\infty}$ page is free as a module over the Green functor $k^{\bigstar}$, hence there can't be any extension problems and we get the module structure:
\begin{equation}
k^{\bigstar}(B_{C_2}\Sigma_{2+})=k^{\bigstar}\{e^{j\rho},e^{j\rho+\sigma}\}
\end{equation}	 
It's easier to prove (using the homotopy fixed point spectral sequence) that:
\begin{equation}
k^{h\bigstar}(B_{C_2}\Sigma_{2+})=k^{h\bigstar}[w]
\end{equation}	 
where $w$ has cohomological degree $1$. The map $k\to k^h$ from section \ref{Notations} induces
\begin{equation}
k^{\bigstar}(B_{C_2}\Sigma_{2+})\to k^{h\bigstar}(B_{C_2}\Sigma_{2+})
\end{equation}	 
which is localization with $u_{\sigma_2}$ being inverted. Thus we can see that $c=e^{\sigma}$ maps to $u_{\sigma_2}w$ (or $a_{\sigma_2}+u_{\sigma_2}w$), $b=e^{\rho}$ maps to $a_{\sigma_2}w+u_{\sigma_2}w^2$ and conclude that:
\begin{equation}
k_{C_2}^{\bigstar}(B_{C_2}\Sigma_{2+})=k_{C_2}^{\bigstar}[c,b]/(c^2=a_{\sigma_2}c+u_{\sigma_2}b)
\end{equation}
$B_{C_2}\Sigma_2$ is a $C_2$-$H$-space so $k^{\bigstar}_{C_2}(B_{C_2}\Sigma_{2+})$ is a Hopf algebra (since it is flat over $k^{\bigstar}_{C_2}$). For degree reasons, we can see that
\begin{gather}
\Delta(c)=c\otimes 1+1\otimes c\\
\Delta(b)=b\otimes 1+1\otimes b\\
\epsilon(c)=\epsilon(b)=0
\end{gather}
(we can add $a_{\sigma_2}$ to $c$ to force $\epsilon(c)=0$). The primitive elements are spanned by $c, b^{2^i}$.

 \section{\texorpdfstring{The cohomology of $B_{C_4}\Sigma_2$}{The cohomology of the C4 classifying space of O(1)}}\label{BC4Sigma2Summary}
 
 In the next section we shall construct a cellular decomposition of $B_{C_4}\Sigma_2$ giving rise to a spectral sequence computing $k^{\bigstar}(B_{C_4}\Sigma_{2+})$. Here's the result of the computation, describing $k^{\bigstar}(B_{C_4}\Sigma_{2+})$ as a Green functor algebra over $k^{\bigstar}$:
 
 \begin{prop}\label{BiggestProposition}There exist elements $e^a,e^u,e^{\lambda},e^{\rho}$ in degrees $\sigma+\lambda,\sigma+\lambda-2, \lambda, \rho$ of $k^{\bigstar}_{C_4}(B_{C_4}\Sigma_{2+})$ respectively, such that 
 	\begin{equation}
 	k^{\bigstar}_{C_4}(B_{C_4}\Sigma_{2+})=\frac{k_{C_4}^{\bigstar}\Big[e^a,\dfrac{e^u}{u_{\sigma}^i},\dfrac{e^{\lambda}}{u_{\sigma}^i},e^{\rho}\Big]_{i\ge 0}}S
 	\end{equation}
The relation set $S$ consists of two types of relations (we use indices $i,j\ge 0$):
 	\begin{itemize}
 		\item Module relations:
    \begin{gather}
	\frac{a_{\sigma}^2}{a_{\lambda}^j}\frac{e^{u}}{u_{\sigma}^i}=0\\
\frac{\frac{\theta}{a_{\lambda}}a_{\sigma}}{u_{\sigma}^{i-2}a_{\lambda}^{j-1}}e^{a}+\frac{s}{u_{\sigma}^{i-1}a_{\lambda}^{j-2}}e^{u}=\frac{a_{\sigma}^2}{a_{\lambda}^j}\frac{e^{\lambda}}{u_{\sigma}^i}
\end{gather}
	\item Multiplicative relations:
    \begin{gather}
    	    \frac{e^u}{u_{\sigma}^i}\frac{e^u}{u_{\sigma}^j}
    	=\frac{u_{\lambda}}{u_{\sigma}^{i+j-2}}e^{\lambda}\\
    	\frac{e^{\lambda}}{u_{\sigma}^i}\frac{e^u}{u_{\sigma}^j}
    	=\frac{u_{\lambda}}{u_{\sigma}^{i+j}}e^a+a_{\lambda}\frac{e^u}{u_{\sigma}^{i+j}}\\
    	e^a\frac{e^u}{u_{\sigma}^i}
    	=\frac{u_{\lambda}}{u_{\sigma}^{i-1}}e^{\rho}+a_\sigma \frac{u_{\lambda}}{u_{\sigma}^i}e^a\\
    	\frac{e^{\lambda}}{u_{\sigma}^i}\frac{e^{\lambda}}{u_{\sigma}^j}
    	=\frac{u_{\lambda}}{u_{\sigma}^{i+j+1}}e^{\rho}+a_\sigma\frac{u_{\lambda}}{u_{\sigma}^{i+j+2}} e^a+a_{\lambda}\frac{e^{\lambda}}{u_{\sigma}^{i+j}}\\
    	e^a\frac{e^{\lambda}}{u_{\sigma}^i}
    	=\frac{e^u}{u_{\sigma}^{i+1}}e^{\rho}+a_\sigma \frac{u_\lambda}{u_{\sigma}^{i+1}}e^{\rho} \\
    	(e^a)^2=
    	u_{\sigma}e^{\lambda}e^{\rho}+a_\sigma \frac{e^u}{u_\sigma}e^\rho +u_{\sigma}a_{\lambda}e^{\rho}+a_\sigma a_\lambda e^a
 \end{gather}
 	\end{itemize}
 The middle level of $k^{\bigstar}(B_{C_4}\Sigma_{2+})$ is generated by the restrictions of $e^a,e^u,e^\lambda,e^\rho$, which we denote by $\bar e^a,\bar e^u,\bar e^\lambda,\bar e^\rho$ respectively, and two quotients as follows:
  	\begin{equation}
 	k^{\bigstar}_{C_2}(B_{C_4}\Sigma_{2+})=\frac{k_{C_2}^{\bigstar}\Big[\bar e^a,\bar e^u,\bar e^{\lambda}, \bar e^\rho, \dfrac{\sqrt{\bar a_\lambda \bar u_\lambda}\bar e^u}{ \bar u_\lambda},\dfrac{\bar a_\lambda \bar u_\sigma^{-1}\bar e^u+\sqrt{\bar a_\lambda \bar u_\lambda}\bar e^{\lambda}}{\bar u_\lambda}\Big]}{\Res^4_2(S)}
 \end{equation}
Here, $\Res^4_2(S)$ denotes the relation set obtained by applying the ring homomorphism $\Res^4_2$ on each relation of $S$. That is, we have the module relations for any $i\ge 0$:
 \begin{equation}
 	\frac{v}{\bar a_\lambda^i}\bar e^u=\frac{v}{\bar a_\lambda^i}\bar e^\lambda=0
 \end{equation}
and the multiplicative relations:
   \begin{gather}
	(\bar e^u)^2=\bar u_\sigma^2\bar u_{\lambda}\bar e^\lambda\\
   	\bar e^{\lambda}\bar e^u=\bar u_\lambda\bar e^a+\bar a_{\lambda}\bar e^u\\
\bar e^a\bar e^u
=\bar u_{\lambda}\bar u_{\sigma}\bar e^{\rho}\\
(\bar e^{\lambda})^2 =\bar u_{\lambda}\bar u_{\sigma}^{-1}
\bar e^{\rho}+\bar a_{\lambda}\bar e^{\lambda}\\
\bar e^a\bar e^{\lambda}=\bar u_{\sigma}^{-1}\bar e^u\bar e^{\rho}\\
(\bar e^a)^2=
\bar u_{\sigma}\bar e^{\lambda}\bar e^{\rho}+\bar u_{\sigma}\bar a_{\lambda}\bar e^{\rho}
\end{gather}

  As for the Mackey functor structure, the Weyl group $C_4/C_2$ action on the generators is trivial and we have:
 \begin{itemize}
 	\item Mackey Functor relations: \begin{gather}
 		\Tr_2^4\Big(\bar u_{\sigma}^{-i}\dfrac{\sqrt{\bar a_\lambda \bar u_\lambda}\bar e^u}{\bar u_\lambda}\Big)=a_{\sigma}\frac{e^{u}}{u_{\sigma}^{i+1}}\\
 		\Tr_2^4\Big(\bar u_{\sigma}^{-i}\dfrac{\bar a_\lambda \bar u_\sigma^{-1}\bar e^u+\sqrt{\bar a_\lambda \bar u_\lambda}\bar e^{\lambda}}{\bar u_\lambda}\Big)=a_{\sigma}\frac{e^{\lambda}}{u_{\sigma}^{i+1}}
 	\end{gather}
 \end{itemize}
Finally, the bottom level is
 	\begin{equation}
	k^{\bigstar}_e(B_{C_4}\Sigma_{2+})=k_e^{\bigstar}[\Res^4_1(e^u)]
\end{equation}
with trivial Weyl group $C_4$ action and Mackey functor relations obtained by applying $\Res^4_1$ to the multiplicative relations of $S$:
   \begin{gather}
	\Res^4_1 e^\lambda=\overline{\bar u}_\sigma^{-2}\overline{\bar u}_{\lambda}^{-1}\Res^4_1(e^u)^2\\
	\Res^4_1 e^a=\overline{\bar u}_\sigma^{-2}\overline{\bar u}_\lambda^{-2}\Res^4_1(e^u)^3\\
\Res^4_1 e^{\rho}=\overline{\bar u}_{\lambda}^{-3}\overline{\bar u}_{\sigma}^{-3}	\Res^4_1(e^u)^4
\end{gather}
 \end{prop}
Note: For every quotient $y/x$ there is a defining relation $x\cdot (y/x)=y$. We have omitted these implicit module relations from the description above.\medbreak

The best description of the middle level is in terms of the generators $c,b$ of $$k^{\bigstar'}_{C_2}(B_{C_2}\Sigma_{2+})=\frac{k^{\bigstar'}_{C_2}[c,b]}{c^2=a_{\sigma_2}c+u_{\sigma_2}b}$$
Here, $\bigstar'$ ranges in $RO(C_2)$ and to get $k^{\bigstar}_{C_2}(B_{C_4}\Sigma_{2+})$ for $\bigstar$ in $RO(C_4)$, we have to restrict to $RO(C_2)$ representations of the form $n+2m\sigma_2$. In this way, $$k^{\bigstar}_{C_2}(B_{C_4}\Sigma_{2+})=k^{\bigstar'}_{C_2}(B_{C_2}\Sigma_{2+})[\bar u_\sigma^{\pm}]$$
where $\bigstar=n+m\sigma+k\lambda$ in $RO(C_4)$ corresponds to $\bigstar'=n+m+2k\sigma_2$ in $RO(C_2)$. The correspondence of generators is:
\begin{align}
	\bar e^a&=\bar u_{\sigma}(a_{\sigma_2}b+bc)\\
	\bar e^u&=\bar u_{\sigma}u_{\sigma_2}c\\
	\bar e^\lambda&=c^2\\
	\bar e^\rho&=\bar u_{\sigma}b^2
\end{align}

We can also express the map to homotopy fixed points in terms of our generators:
 
  \begin{prop}There is a choice of the degree $1$ element $w$ in 
 	\begin{equation}
 		k^{hC_4\bigstar}(B_{C_4}\Sigma_{2+})=k^{hC_4\bigstar}[w]
 	\end{equation}
so that the localization map $k^{\bigstar}_{C_4}(B_{C_4}\Sigma_{2+})\to	k^{hC_4\bigstar}(B_{C_4}\Sigma_{2+})$ induced by $k\to k^h$ and inverting $u_\sigma,u_\lambda$ is:
	\begin{align}
	e^u&\mapsto u_{\sigma}u_{\lambda}w\\
	e^{\lambda}&\mapsto u_{\lambda}w^2\\
	e^a&\mapsto u_{\sigma}u_{\lambda}w^3+u_{\sigma}a_{\lambda}w\\
	e^{\rho}&\mapsto u_{\sigma}u_{\lambda}w^4+a_\sigma u_\lambda w^3+u_{\sigma}a_{\lambda}w^2+a_\sigma a_\lambda w
\end{align}
\end{prop}
 
 \begin{prop}\label{NonFlatnessIs}
 	The module $k^{\bigstar}(B_{C_4}\Sigma_{2+})$ is not flat over $k^{\bigstar}$.\end{prop}
 
 \begin{proof}Let $R=k^{\bigstar}$ and $M=k^{\bigstar}(B_{C_4}\Sigma_{2+})$. Consider the map $f:R\to \Sigma^{2\sigma-\lambda}R$ given on top level by multiplication with $a_{\sigma}^2/a_{\lambda}$ and determined on the lower levels by restricting (so it's multiplication with $v\bar u_{\sigma}^2$ on the middle level and $0$ on the bottom level). If $M$ is a flat $R$-module then we have an exact sequence
 	\begin{equation}
 	0\to M\boxtimes_R Ker(f) \to M\xrightarrow{f}\Sigma^{2\sigma-\lambda}M
 	\end{equation}
 	The restriction functor $\Res_2^4$ from $R$ modules to $\Res_2^4R$ modules is exact and symmetric monoidal, so we replace $M,R, Ker(f)$ by $\Res^4_2M, \Res^4_2R, \Res^4_2Ker(f)$ respectively and have an exact sequence of $C_2$ Mackey functors. Using the notation involving the $C_2$ generators $c,b$ and writing $a=a_{\sigma_2}, u=u_{\sigma_2}$, we have $M=\oplus_{i\ge 0}R\{b^{2i},cb^{2i+1}\}\oplus \oplus_{i\ge 0} R\{ab^{2i+1},ub^{2i+1},acb^{2i},ucb^{2i}\}/\sim$. The map $f$ maps each summand to itself, so we may replace $M$ by $R\{c,ab,ub,acb,ucb\}/\sim$ and continue to have the same exact sequence as above. The top level then is:
 	\begin{equation}
 	0\to (M\boxtimes_R Ker(f))(C_2/C_2) \to M(C_2/C_2) \xrightarrow{v\bar u_{\sigma}^2}M(C_2/C_2) 
 	\end{equation}
 	and $v$ acts trivially on $ab,ub,ac,uc$ i.e. on $M(C_2/C_2)$ so we get
 	\begin{equation}
 	(M\boxtimes_RKer(f))(C_2/C_2)=M(C_2/C_2) 
 	\end{equation}
 	We compute directly from definition that $(M\boxtimes_RKer(f))(C_2/C_2) $ is isomorphic to $M(C_2/C_2) \otimes_{R(C_2/C_2) }I$ where $I:=Ker(R\xrightarrow{v}R)$. But $M(C_2/C_2) \otimes_{R(C_2/C_2) }I\to M(C_2/C_2) $ has image $IM(C_2/C_2) $ hence 
 	\begin{equation}IM(C_2/C_2) =M(C_2/C_2) \end{equation}
 	This contradicts that $ab=e^{'\lambda+1}$ is not divisible by any element of the ideal $I$ ($e^{'\lambda+1}$ is only divisible by $\bar u_{\sigma}^{\pm i}\in R$ which are not in $I$).
 \end{proof}
 \section{\texorpdfstring{A cellular decomposition of $B_{C_4}\Sigma_2$}{A cellular decomposition of the C4 classifying space of O(1)}}\label{BC4S2dec}
 
We denote the generators of $C_4$ and $\Sigma_2$ by $g$ and $h$ respectively; let also $\tau$ be the sign representation of $\Sigma_2$ and $\rho=1+\sigma+\lambda$ the regular representation of $C_4$.
 
 The graph subgroups of $C_4\times \Sigma_2$ are $C_4=\ev{g},C_2=\ev{g^2},\Delta=\ev{gh},\Delta'=\ev{g^2h},e$. 
 
 We thus have a model for the universal space:
 \begin{equation}
 E_{C_4}\Sigma_2=S(\infty(\rho\otimes \tau))
 \end{equation}
 and $B_{C_4}\Sigma_2$ is $\mathbb R P^{\infty}$ with nontrivial $C_4$ action:
 \begin{equation}
 g(x_1,x_2,x_3,x_4,...)=(x_1,-x_2,-x_4,x_3,...)
 \end{equation}
 
 $S(\infty (\rho\otimes \tau))$ is the space
 \begin{equation}
 S(\infty)=\{(x_n):\text{ finitely supported and }\sum_ix_i^2=1\}
 \end{equation}
 with $C_4\times \Sigma_2$ action
 \begin{gather}
 g(x_1,x_2,x_3,x_4,x_5,...)=(x_1,-x_2,-x_4,x_3,x_5,...)\\
 h(x_1,x_2,...)=(-x_1,-x_2,...)
 \end{gather}

We shall use the notation $(x_1,...,x_n)$ for the point $(x_1,...,x_n,0,0,...)\in S(\infty)$. Moreover, the subspace of $S(\infty)_+$ where only $x_1,...,x_n$ are allowed to be nonzero shall be denoted by $\{(x_1,...,x_n)\}$.

We now describe a cellular decomposition of $E_{C_4}\Sigma_{2+}$ where the orbits are $C_4\times \Sigma_{2+}/H \wedge S^V$ where $V$ is a $C_4$ representation.
 \begin{itemize}
 	\item Start with $\{(x_1)\}$ the union of two points $(1),(-1)$ and the basepoint. This is $C_4\times \Sigma_2/C_{4+}$.
 	\item $\{(x_1)\}$ includes in $\{(x_1,x_2)\}=S(1+\sigma)_+$. The cofiber is the wedge of two circles, corresponding to $x_2$ being positive or negative, and the action is
 	\begin{equation}
 	g(x_1,+)=(x_1,-)\text{ , }h(x_1,+)=(-x_1,-)
 	\end{equation}
 	After applying the self equivalence given by $f(x_1,+)=(x_1,+)$ and $f(x_1,-)=(-x_1,-)$, the action becomes
 	 	\begin{equation}
 	g(x_1,+)=(-x_1,-)\text{ , }h(x_1,+)=(x_1,-)
 	\end{equation}
 	This is exactly $C_4\times \Sigma_2/\Delta_+\wedge S^{\sigma}$.

 	\item $\{(x_1,x_2)\}$ includes in $\{(x_1,x_2,x_3,0),(x_1,x_2,0,x_4)\}$; the cofiber is the wedge of four spheres corresponding to the sign of the nonzero coordinate among the last two coordinates. If we number the spheres from $1$ to $4$ and use $(x,y)^i$ coordinates to denote them $i=1,2,3,4$ then
 	\begin{equation}
 	g(x,y)^i=(x,-y)^{i+1}\text{ , }h(x,y)^i=(-x,-y)^{i+2}
 	\end{equation}
 	Applying the self equivalence
 	\begin{equation}
 	f(x,y)^1=(x,y)^1\text{ , }f(x,y)^2=(-y,x)^2\text{ , }f(x,y)^3=(-x,-y)^3\text{ , }f(x,y)^4=(y,-x)^4
 	\end{equation}
 	the action becomes $g(x,y)^i=(-y,x)^{i+1}$ and $h(x,y)^i=(x,y)^{i+2}$ i.e. we have  $C_4\times \Sigma_2/\Delta'_+\wedge S^{\lambda}$.

 	\item $\{(x_1,x_2,x_3,0),(x_1,x_2,0,x_4)\}$ includes in $\{(x_1,x_2,x_3,x_4)\}=S(\rho\otimes \tau)$ and the cofiber is the wedge of four $S^3$'s corresponding to the signs of $x_3,x_4$. Analogously to the item above, we get the space $C_4\times \Sigma_2/\Delta'_+\wedge S^{1+\lambda}$.
 	
 	\item The process now repeats:  $\{(x_1,x_2,x_3,x_4)\}$ includes in $\{(x_1,x_2,x_3,x_4,x_5)\}$ and the cofiber is the wedge of two $S^4$'s corresponding to the sign of $x_5$ and we get $C_4\times \Sigma_2/C_{4+}\wedge S^{1+\sigma+\lambda}$. And so on...
 	
 \end{itemize}  
 
 We get the decomposition of $B_{C_4}\Sigma_{2+}$ where the associated graded is:
 \begin{gather}
 gr_{4j}=S^{j\rho}\\
 gr_{4j+1}=S^{j\rho+\sigma}\\
 gr_{4j+2}=\Sigma^{j\rho+\lambda}C_4/C_{2+}\\
 gr_{4j+3}=\Sigma^{j\rho+1+\lambda}C_4/C_{2+}
 \end{gather}
 This filtration gives a spectral sequence of $k^{\bigstar}$ modules converging to $k^{\bigstar}(B_{C_4}\Sigma_{2+})$ that we shall analyze in the next section.
 
 \subsection{A decomposition using trivial spheres}\label{CWDecompOrbits}
 The cellular decomposition of $B_{C_4}\Sigma_{2}$ we just established, consists of one cell in every dimension, whereby "cell" we mean a space of the form $(C_4/H)_+\wedge S^V$ where $H$ is a subgroup of $C_4$ and $V$ is a real non-virtual $C_4$-representation; let us call this a "type I" decomposition. It is also possible to obtain a decomposition using only "trivial spheres", namely with cells of the form $(C_4/H)_+\wedge S^n$; we shall refer to this as a "type II" decomposition. A type I decomposition can be used to produce a type II decomposition by using the type II decompositions of each type I cell $(C_4/H)_+\wedge S^V$. This is useful for computer-based calculations, since type II decompositions lead to chain complexes as opposed to spectral sequences ($k_*((C_4/H)_+\wedge S^V)$ is concentrated in a single degree if and only if $V$ is trivial). Equipped with a type I decomposition, the computer program of \cite{Geo19} can calculate the additive structure of $k^{\bigstar}(B_{C_4}\Sigma_{2+})$ in a finite range (this can be helpful with our spectral sequence calculations: see Remark \ref{ProgramDifferential}).
 
 We note however that the type II decomposition of $B_{C_4}\Sigma_2$ obtained this way is rather inefficient and not minimal. For example, consider the subspace spanned by homogeneous coordinates $\{(x_1:x_2:x_3:x_4)\}$; this is obtained from the subspace $\{(x_1:x_2:0:x_4),(x_1:x_2:x_3:0)\}$ by attaching a cell $(C_4/C_2)_+\wedge S^{1+\lambda}$. The sphere $S^{1+\lambda}$ itself has a top dimensional cell $C_{4+}\wedge S^3$; combining the two we get $(C_4/C_2)_+\wedge C_{4+}\wedge S^3=(C_{4+}\vee C_{4+})\wedge S^3$ i.e. 2 cells in dimension 3. On the other hand, $\{(x_1:x_2:x_3:x_4)\}$ can also be obtained from the subspace $\{(x_1:x_2:0:x_4),(x_1:x_2:x_3:0),(x_1:0:x_3:x_4)\}$ by attaching $C_{4+}\wedge S^3$ i.e. 1 cell in dimension 3. Working out the lower dimensions, we can see that the type II decomposition of $\R P^{\sigma+\lambda}=\{(x_1:x_2:x_3:x_4)\}$ obtained by expanding the type I decomposition has 12 total cells, while it is possible to obtain a more efficient type I decomposition with 9 total cells.

 \section{\texorpdfstring{The spectral sequence for $B_{C_4}\Sigma_2$}{The spectral sequence for the C4 classifying space of O(1)}}\label{BC4S2ss}
 
 Applying $k^{\bigstar}$ on the filtration of $B_{C_4}\Sigma_{2+}$ gives a spectral sequence \begin{equation}
 	E_1^{s,V}=k^{V}gr_s\implies k^VB_{C_4}\Sigma_{2+}
 \end{equation}
 The differential $d^r$ has $(V,s)$ bidegree $(1,r)$ so it goes $1$ unit to the right and $r$ units up in $(V,s)$ coordinates.\medbreak
 
 Before we can write down the $E_1$ page, we will need some notation: For a $G$-Mackey functor $M$ and subgroup $H\subseteq G$, $M_{G/H}$ denotes the $G$-Mackey functor defined on orbits as $M_{G/H}(G/K)=M(G/H\times G/K)$; the restriction, transfer and Weyl group action in $M_{G/H}$ are induced from those in $M$.  For $G=C_4$ and $H=C_2$, the bottom level of $M_{C_4/C_2}$ is: $$M_{C_4/C_2}(C_4/e)=M(C_4/e\times C_4/C_2)=M(C_4/e)\oplus M(C_4/e)=M(C_4/e)\{x,y\}$$
 where $x,y$ are used to distinguish the two copies of $M(C_4/e)$, i.e. so that any element of $M_{C_4/C_2}(C_4/e)$ can be uniquely written as $mx+m'y$ for $m,m'\in M(C_4/e)$. The Weyl group $W_{C_4}e=C_4$ acts as
 $$g(mx+m'y)=(gm)(gx)+(gm')(gy)=(gm)y+(gm')x$$
 i.e. $y=gx$ for a fixed generator $g\in C_4$.
 
 We can then describe $M_{C_4/C_2}$ in terms of $M$ and the computation of the restriction and transfer on $x$, which are shown in the following diagram:
 \begin{equation}
 	M_{C_4/C_2}=\begin{tikzcd}
 		M(C_4/C_2)\{x+gx\}\ar[d, "x+gx\mapsto x+gx" left, bend right]\\
 		M(C_4/C_2)\{x,gx\}\ar[u, "x\mapsto x+gx" right,bend right]\ar[d, "x\mapsto x" left, bend right]\ar[loop right, "C_4/C_2" right]\\
 		M(C_4/e)\{x,gx\}\ar[u, "x\mapsto 0" right,bend right]\ar[loop right, "C_4" right]
 	\end{tikzcd}
 \end{equation}
If $M=R$ is a Green functor, then $R_{C_4/C_2}$ is an $R$-module. Its top level, namely $R(C_4/C_2)\{x+gx\}$, is an $R(C_4/C_4)$ module via extension of scalars along the restriction map $\Res^4_2:R(C_4/C_4)\to R(C_4/C_2)$.

 \subsection{\texorpdfstring{The $E_1$ page}{The E1 page}} \label{E1}
 The rows in the $E_1$ page are:
 \begin{gather}
 E_1^{\bigstar,4j}=k^{\bigstar-j\rho}\\
 E_1^{\bigstar,4j+1}=k^{\bigstar-j\rho-\sigma}\\
 E_1^{\bigstar,4j+2}=(k^{\bigstar-j\rho-\lambda})_{C_4/C_2}\\
 E_1^{\bigstar,4j+3}=(k^{\bigstar-j\rho-\lambda-1})_{C_4/C_2}
 \end{gather}

 We will write $e^{j\rho},e^{j\rho+\sigma},e^{j\rho+\lambda},e^{j\rho+\lambda+1}$ for the unit elements corresponding to the $E_1$ terms above, living in degrees $V=j\rho,j\rho+\sigma,j\rho+\lambda,j\rho+\lambda+1$ and filtrations $s=4j,4j+1,4j+2,4j+3$ respectively. We also write $\bar e^V, \overline{\bar e}^V$ for their restrictions to the middle and bottom levels respectively. In this way: 
\begin{equation}
E_1^{\bigstar,*}=k^{\bigstar}\{e^{j\rho},e^{j\rho+\sigma}\}\oplus (k^{\bigstar})_{C_4/C_2}\{e^{j\rho+\lambda},e^{j\rho+\lambda+1}\}
\end{equation} 
and the three levels of the Mackey functor $E_1^{\bigstar,*}$, from top to bottom, are:
\begin{gather}
k^{\bigstar}_{C_4}\{e^{j\rho},e^{j\rho+\sigma}\}\oplus k^{\bigstar}_{C_2}\{e^{j\rho+\lambda}(x+gx),e^{j\rho+\lambda+1}(x+gx)\} \\
k^{\bigstar}_{C_2}\{\bar e^{j\rho},\bar e^{j\rho+\sigma}\}\oplus k^{\bigstar}_{C_2}\{\bar e^{j\rho+\lambda}x, \bar e^{j\rho+\lambda}gx, \bar e^{j\rho+\lambda+1}x, \bar e^{j\rho+\lambda+1}gx\}\\
k^{\bigstar}_{e}\{\overline{\bar e}^{j\rho},\overline{\bar e}^{j\rho+\sigma}\}\oplus k^{\bigstar}_e\{\overline{\bar e}^{j\rho+\lambda}x, \overline{\bar e}^{j\rho+\lambda}gx, \overline{\bar e}^{j\rho+\lambda+1}x, \overline{\bar e}^{j\rho+\lambda+1}gx\}
\end{gather}
For the top level, $k_{C_2}^{\bigstar}$ is a $k_{C_4}^{\bigstar}$ module through the restriction $\Res^4_2:k_{C_4}^{\bigstar}\to k_{C_2}^{\bigstar}$:
 \begin{equation}
 k^{\bigstar}_{C_2}=\frac{k^{\bigstar}_{C_4}[u_{\sigma}^{-1}]}{a_{\sigma}}\{1,\sqrt{\bar a_{\lambda}\bar u_{\lambda}}\}
 \end{equation}
It's important to note that this is \emph{not} a cyclic $k_{\bigstar}^{C_4}$ module.\medbreak
 
At this point, the reader may want to look over pictures of the $E_1$ page that we have included in the Appendix \ref{AppendixSS}.

\subsection{\texorpdfstring{The $d^1$ differentials}{The d1 differentials}}\label{d1comp} In this subsection, we explain how the $d^1$ differentials on each level are computed. We shall need this crucial remark:

\begin{rem}\label{Remark}The restriction of the $C_4$ action on $B_{C_4}\Sigma_2$ to $C_2\subseteq C_4$ results in a $C_2$ space equivalent to $B_{C_2}\Sigma_2$. The equivariant cohomology of this space is known from subsection \ref{BC2S2} and we shall use this result to compute the middle level spectral sequence for $B_{C_4}\Sigma_2$. Further restricting to the trivial group $e\subseteq C_4$, we get the nonequivariant space $\R P^{\infty}$ and this will be used to compute the bottom level spectral sequence.
\end{rem}

First of all, the bottom level spectral sequence is concentrated on the diagonal and the nontrivial $d^1$'s are $k\{x,gx\}\to k\{x,gx\}$, $x\mapsto x+gx$ (since $k^*(\R P^{\infty})$ is $k$ in every nonnegative degree).\medbreak

 The $d^1$'s on middle and top level are computed from the fact that they are $k^{\bigstar}$ module maps, hence determined on $$e^{j\rho},e^{j\rho+\sigma},\bar u_\sigma^{-i}\sqrt{\bar a_\lambda\bar u_\lambda}^\epsilon e^{j\rho+\lambda+\epsilon'}(x+gx)$$
 for the top level ($\epsilon,\epsilon'=0,1$), and on $$\bar e^{j\rho},\bar e^{j\rho+\sigma},\bar e^{j\rho+\lambda}x,\bar e^{j\rho+\lambda+1}x$$ for the middle level. We remark that because $k^{\bigstar}_{C_2}$ is not a cyclic $k^{\bigstar}_{C_4}$ module, it does not suffice to compute the top level $d^1$ on $e^{j\rho},e^{j\rho+\sigma},e^{j\rho+\lambda},e^{j\rho+\lambda+1}$.\medbreak

 The $d^1$ differentials from row $4j$ to row $4j+1$ are all determined by the differential 
 $d^1:ke^{j\rho}\to k^{1-\sigma}e^{j\rho+\sigma}$. Note that $k^{1-\sigma}$ is generated by $0|u_{\sigma}^{-1}|\overline{\bar u}_{\sigma}^{-1}$ (this notation was defined in \cite{Geo19} and expresses the generators of all three levels from top to bottom separated by vertical columns). The $d^1$ is trivial on bottom level, and using the fact that it commutes with restriction we can see that it's trivial in all levels. \medbreak
 
 Similarly, the $d^1$ differentials from row $4j+1$ to row $4j+2$ are all determined by $d^1:ke^{j\rho+\sigma}\to (k^{\sigma-\lambda+1})_{C_4/C_2}e^{j\rho+\lambda}$. Note that $(k^{\sigma-\lambda+1})_{C_4/C_2}$ is generated by $$v\bar u_{\sigma}(x+gx)|v\bar u_{\sigma}(x,gx)|\overline{\bar u}_{\sigma}\overline{\bar u}_{\lambda}^{-1}(x,gx)$$ The differential is trivial on the bottom level, but on middle level the $C_2$ computation gives $k^{\sigma}_{C_2}(B_{C_4}\Sigma_{2+})=0$ forcing the differential to be nontrivial (the only other way to kill $E_1^{\sigma-\lambda+1,4j+2}(C_4/C_2)=k^2$ is for the $d^1$ differential from row $4j+2$ to $4j+3$ to be the identity $k^2\to k^2$ on middle level, which can't happen as we show in the next paragraph). Thus:
 \begin{gather}
 d^1(\bar e^{j\rho+\sigma})=v\bar u_{\sigma}\bar e^{j\rho+\lambda}(x+gx)\\
 d^1(e^{j\rho+\sigma})=v\bar u_{\sigma}e^{j\rho+\lambda}(x+gx)
 \end{gather}
The $d^1$ differentials from row $4j+2$ to row $4j+3$ 
are determined by \begin{gather}
	d^1:k_{C_4/C_2}\bar u_\sigma^{-i}e^{j\rho+\lambda}\to k_{C_4/C_2}\bar u_\sigma^{-i}e^{j\rho+\lambda+1}\text{ and }\\
	d^1:k_{C_4/C_2}\bar u_\sigma^{-i}\sqrt{\bar a_{\lambda}\bar u_{\lambda}}e^{j\rho+\lambda}\to k_{C_4/C_2}\bar u_\sigma^{-i}\sqrt{\bar a_{\lambda}\bar u_{\lambda}}e^{j\rho+\lambda+1}
\end{gather}
On bottom level, these $d^1$'s all are $x\mapsto x+gx$ and the commutation with restriction and transfer gives:
\begin{gather}
d^1(\bar e^{j\rho+\lambda}x)=\bar e^{j\rho+\lambda+1}(x+gx)\\
d^1(\bar u_\sigma^{-i}e^{j\rho+\lambda}(x+gx))=0\\
d^1(\bar u_\sigma^{-i}\sqrt{\bar a_{\lambda}\bar u_{\lambda}} e^{j\rho+\lambda}(x+gx))=0
\end{gather}
 
Finally, the $d^1$ differentials from row $4j+3$ to row $4j+4$ are determined by
\begin{gather}
	d^1:k_{C_4/C_2}\bar u_\sigma^{-i}e^{j\rho+\lambda+1}\to k^{1-\sigma}e^{j\rho+\rho}\text{ and }\\
d^1:k_{C_4/C_2}\bar u_\sigma^{-i}\sqrt{\bar a_{\lambda}\bar u_{\lambda}}e^{j\rho+\lambda+1}\to k^{2-\sigma+\lambda}e^{j\rho+\rho}
\end{gather}
These are trivial on the bottom level and by the commutation with restriction and transfer we can see that they are trivial on all levels.\medbreak

This settles the $E_1$ page computation.

\subsection{Bottom level computation}

We can immediately conclude that the bottom level spectral sequence collapses in $E_2$, giving a single $k$ in every $RO(C_4)$ degree. Thus there are no extension problems and the $C_4$ (Weyl group) action is trivial.
 
\subsection{Middle level computation}\label{MidLevelComp}

By remark \ref{Remark} we can immediately conclude that the middle level spectral sequence collapses on $E_2=E_\infty$. 

If we have a middle level element $\alpha\in E_\infty^{s,V}$ and $E_\infty^{t,V}=0$ for $t>s$ then $\alpha$ lifts uniquely to $k^{\bigstar}_{C_2}(B_{C_4}\Sigma_{2+})$. If on the other hand $E_\infty^{t,V}\neq 0$ for some $t>s$, then there are multiple lifts of $\alpha$. In that case, we pick the lift for which there are no exotic restrictions (if possible). For example, if $\Res^2_1(\alpha)=0$ in $E_\infty$ and there is a unique lift $\beta$ of $\alpha$ such that $\Res^2_1(\beta)=0$, then we use $\beta$ as our lift of $\alpha$.

 With this in mind, and the computation of the $E_2$ terms, we have:
\begin{itemize}
	\item The elements $\bar e^{j\rho}$ survive the spectral sequence and lift uniquely to elements $\bar e^{j\rho}$ in $k^{\bigstar}_{C_2}(B_{C_4}\Sigma_{2+})$.
	\item  The elements $\bar u_{\sigma}^i\bar e^{j\rho+\sigma}$ don't survive, but every other multiple of $\bar e^{j\rho+\sigma}$ does. These multiples are generated by: $$\bar a_\lambda \bar e^{j\rho+\sigma}, \sqrt{\bar a_{\lambda}\bar u_{\lambda}}\bar e^{j\rho+\sigma},\bar u_{\lambda}\bar e^{j\rho+\sigma}$$ 
	\begin{itemize}
		\item The elements $\bar u_{\lambda}\bar e^{j\rho+\sigma}$ lift uniquely to elements $\bar e^{j,u}$ in $k^{\bigstar}_{C_2}(B_{C_4}\Sigma_{2+})$.
		\item For each $j\ge 0$, the element $\bar a_\lambda \bar e^{j\rho+\sigma}$ has two distinct lifts. On $E_\infty$ we have that  $\Res^2_1(\bar a_\lambda \bar e^{j\rho+\sigma})=0$ and on $k^{\bigstar}_{C_2}(B_{C_4}\Sigma_{2+})$ only one of the two lifts has trivial restriction. We denote that lift by $\tilde e^{j,a}$.
		\item Similarly, the elements $\sqrt{\bar a_{\lambda}\bar u_{\lambda}}\bar e^{j\rho+\sigma}$ have trivial restriction on $E_\infty$ and unique lifts with trivial restriction on $k^{\bigstar}_{C_2}(B_{C_4}\Sigma_{2+})$, that we denote by $e^{j,au}$. 	\end{itemize}
	\item  The elements $\bar e^{j\rho+\lambda}x, \bar e^{j\rho+\lambda}gx$ don't survive while the elements $\bar e^{j\rho+\lambda}(x+gx)$ do. They lift uniquely to elements in $k^{\bigstar}_{C_2}(B_{C_4}\Sigma_{2+})$ that we denote by $\bar e^{j\rho+\lambda}$. We have the relation 
	\begin{equation}
		v\bar e^{j\rho+\lambda}=0
	\end{equation}
	\item The elements $\bar e^{j\rho+\lambda+1}(x+gx)$ don't survive while the elements $\bar e^{j\rho+\lambda+1}x=\bar e^{j\rho+\lambda+1}gx$ do. They lift uniquely to elements in $k^{\bigstar}_{C_2}(B_{C_4}\Sigma_{2+})$ that we denote by $e^{'j\rho+\lambda+1}$.
\end{itemize}

\begin{rem}
We should explain the notation used for the generators above. \\
First, the elements $\bar e^{j,u}, \bar e^{j\rho+\lambda}$ will turn out to be the restrictions of top level elements $e^{j,u}, e^{j\rho+\lambda}$ respectively, both in $E_{\infty}$ and in $k^{\bigstar}_{C_4}(B_{C_4}\Sigma_{2+})$, hence their notation. \\
Second, the elements $e^{j,au}, e^{'j\rho+\lambda+1}$ are never restrictions, neither in $E_{\infty}$ nor in $k^{\bigstar}_{C_4}(B_{C_4}\Sigma_{2+})$, so their notation is rather ad-hoc: the $au$ in $e^{j,au}$ serves as a reminder of the $\sqrt{\bar a_{\lambda}\bar u_{\lambda}}$ in $e^{j,au}=\sqrt{\bar a_{\lambda}\bar u_{\lambda}}\bar e^{j\rho+\sigma}$, while the prime $'$ in $e^{'j\rho+\lambda+1}$ is used to distinguish them from the top level generators $e^{j\rho+\lambda+1}$ that the $e^{'j\rho+\lambda+1}$ transfer to.\\
Finally, the elements $\tilde e^{j,a}$ are restrictions of top level elements $e^{j,a}$ in $E_{\infty}$, but not in $k^{\bigstar}_{C_4}(B_{C_4}\Sigma_{2+})$ due to nontrivial Mackey functor extensions (exotic restrictions). That's why we denote them by $\tilde e^{j,a}$ as opposed to $\bar e^{j,a}$; the $\bar e^{j,a}$ are reserved for $\Res^4_2(e^{j,a})=\tilde e^{j,a}+\bar u_\sigma e^{'j\rho+\lambda+1}$ (see Lemma \ref{ResOfea}) .
\end{rem}

For convenience, when $j=0$ we write $\tilde e^a,e^{au}, \bar e^u$ in place of $\tilde e^{0,a}, e^{0,au}, \bar e^{0,u}$ respectively.\medbreak
 
 Now recall that $k^{\bigstar}_{C_2}(B_{C_2}\Sigma_{2+})=k^{\bigstar}_{C_2}\{e^{j\rho_2},e^{j\rho_2+\sigma_2}\}$ (see subsection \ref{BC2S2}). We shall write our middle level $C_4$ generators in terms of the $C_2$ generators.
 
 \begin{prop}We have:
 \begin{align}
 \bar e^{j\rho}&=\bar u_{\sigma}^{j}e^{2j\rho_2}\\
  \bar e^{j,u}&= \bar u_{\sigma}^{j+1}u_{\sigma_2}e^{2j\rho_2+\sigma_2}\\
   e^{j,au}&= \bar u_{\sigma}^{j+1}a_{\sigma_2}e^{2j\rho_2+\sigma_2}\\
 \tilde e^{j,a}&=\bar u_{\sigma}^{j+1}a_{\sigma_2}e^{(2j+1)\rho_2}\\
 \bar e^{j\rho+\lambda}&=\bar u_{\sigma}^{j}a_{\sigma_2}e^{2j\rho_2+\sigma_2}+\bar u_{\sigma}^{j}u_{\sigma_2}e^{(2j+1)\rho_2}\\
  e^{'j\rho+\lambda+1}&= \bar u_{\sigma}^{j}e^{(2j+1)\rho_2+\sigma_2}
 \end{align}
 \end{prop}

 \begin{proof}
 	The map $f:E_{C_4}\Sigma_2\to E_{C_2}\Sigma_2$,
 	$f(x_1,x_2,x_3,x_4,...)=(x_1,x_3,x_2,x_4,...)$ is a $C_2\times \Sigma_2$ equivariant homeomorphism and induces a map on filtrations:
 	\begin{center}
 	 \begin{tikzcd}
 		S(1)\ar[d, <->]\ar[r,hook]&S(1+\sigma)\ar[rd, ]\ar[r,hook]&X\ar[r,hook]\ar[rd]&S(\rho)\ar[d]\ar[r,hook]&\cdots\\
 		S(1)\ar[r,hook]&S(\rho_2)\ar[r,hook]\ar[ru,crossing over]&S(1+\rho_2)\ar[r,hook]\ar[u]&S(2\rho_2)\ar[r,hook]\ar[u]&\cdots
 	\end{tikzcd}
 	\end{center}
 	(the downwards arrows are $f$ while the arrows in the opposite direction are $f^{-1}$).\\ 	
 	To keep the notation tidy, we verify the correspondence of generators for $j=0$. 
 	
 	In the $C_4$ spectral sequence, we have $\tilde e^a\bar u_{\sigma}^{-1}, e^{'\lambda+1}$ in degree $\lambda+1$ and filtrations $1,3$ respectively. In the $C_2$ spectral sequence, we have $a_{\sigma_2}e^{\sigma_2+1}$ and $e^{2\sigma_2+1}$ in the same degree and filtrations $2,3$ respectively. The correspondence of filtrations gives:
 	\begin{gather}e^{'\lambda+1}=e^{2\sigma_2+1}\\
 		\tilde e^a\bar u_{\sigma}^{-1}=a_{\sigma_2}e^{\sigma_2+1}+\epsilon e^{2\sigma_2+1}
 		\end{gather}
 	where $\epsilon=0,1$. Applying restriction on the second equation reveals that $\epsilon=0$ and thus $\tilde e^a\bar u_\sigma^{-1}=a_{\sigma_2}e^{\sigma_2+1}$.
 The correspondence of filtrations in degrees $\lambda-1,\lambda$ gives:
 	\begin{gather}
 	\bar e^u\bar u_{\sigma}^{-1}=\epsilon_1a_{\sigma_2}u_{\sigma_2}+u_{\sigma_2}e^{\sigma_2}\\
 	e^{au}\bar u_{\sigma}^{-1}=\epsilon_2a_{\sigma_2}e^{\sigma_2}+\epsilon_3u_{\sigma_2}e^{\sigma_2+1}\\
 	\bar e^{\lambda}=\epsilon_4a_{\sigma_2}e^{\sigma_2}+\epsilon_5u_{\sigma_2}e^{\sigma_2+1}
 	\end{gather}
 	where $\epsilon_i=0,1$.  Looking at degree $2\lambda+\sigma-2$ in the $C_4$ spectral sequence we see that we have a relation 
 	$$\bar a_{\lambda}\bar e^u=\bar u_{\lambda}\tilde e^a+\epsilon_6\sqrt{\bar a_{\lambda}\bar u_{\lambda}}\bar u_{\sigma}\bar e^{\lambda}+\epsilon_7\bar u_{\sigma}\bar u_{\lambda}e^{'\lambda+1}$$
 	where again $\epsilon_i=0,1$. Combining the equations above we conclude that $$\bar a_{\lambda}\bar e^u=\bar u_{\lambda}\tilde e^a+\sqrt{\bar a_{\lambda}\bar u_{\lambda}}\bar u_{\sigma}\bar e^{\lambda}$$ and
 	\begin{gather}
 	\bar e^u\bar u_{\sigma}^{-1}=u_{\sigma_2}e^{\sigma_2}\\
 	 e^{au}\bar u_{\sigma}^{-1}=a_{\sigma_2}e^{\sigma_2}\\
 	\bar e^{\lambda}=a_{\sigma_2}e^{\sigma_2}+u_{\sigma_2}e^{\sigma_2+1}
 	\end{gather}
 \end{proof}
As a corollary we obtain the relations:
\begin{gather}
\bar u_{\lambda}\tilde e^{j,a}=\bar a_{\lambda}\bar e^{j,u}+\sqrt{\bar a_{\lambda}\bar u_{\lambda}}\bar u_{\sigma}\bar e^{j\rho+\lambda}\\
\bar u_{\lambda}e^{j,au}=\sqrt{\bar a_{\lambda}\bar u_{\lambda}}\bar e^{j,u}\\
\sqrt{\bar a_{\lambda}\bar u_{\lambda}}e^{j,au}=\bar a_{\lambda}\bar e^{j,u}\\
\bar a_{\lambda}e^{j,au}=\sqrt{\bar a_{\lambda}\bar u_{\lambda}}\tilde e^{j,a}+\bar a_{\lambda}\bar u_{\sigma}\bar e^{j\rho+\lambda}\\
\frac{v}{\bar a_{\lambda}^i}\bar e^{j\rho+\lambda}=0
\end{gather} 
Therefore, $k^{\bigstar}_{C_2}(B_{C_4}\Sigma_{2+})$ is spanned as a $k^{\bigstar}_{C_2}$ module by $\bar e^{j\rho}, \tilde e^{j,a}, e^{j,au}, \bar e^{j,u}$, $\bar e^{j\rho+\lambda}, e^{'j\rho+\lambda+1}$ under the relations above. The bottom level $k^{\bigstar}_e(B_{C_4}\Sigma_{2+})$ is free on the restrictions of $\bar e^{j\rho},\bar  e^{j,u}, \bar e^{j\rho+\lambda}, e^{'j\rho+\lambda+1}$.\medbreak
 
 The $C_4/C_2$ (Weyl group) action is trivial: The only extension extensions that may arise are $g\tilde e^{j,a}=\tilde e^{j,a}+\epsilon e^{'j\rho+\lambda+1}$ and $ge^{j,au}=e^{j,au}+\epsilon'\bar e^{j\rho+\lambda}$ where $\epsilon,\epsilon'=0,1$; applying restriction shows that $\epsilon=\epsilon'=0$.\medbreak
 
 The cup product structure can be understood in terms of the $C_2$ generators $c,b$ of subsection \ref{BC2S2}. As an algebra, $k^{\bigstar}_{C_2}(B_{C_4}\Sigma_{2+})$ is generated by $\tilde e^a, e^{au}, \bar e^u, \bar e^\lambda, e^{'\lambda+1}, \bar e^\rho$ under multiplicative relations that are implied by the correspondence of generators:
 \begin{align}
 	e^{\rho}&=\bar u_\sigma b^{2}\\
 	\tilde e^a&=\bar u_\sigma a_{\sigma_2}b\\
 	e^{au}&=\bar u_\sigma a_{\sigma_2}c\\
 	\bar e^{u}&=\bar u_{\sigma}u_{\sigma_2}c\\
 	\bar e^{\lambda}&=c^2=a_{\sigma_2}c+u_{\sigma_2}b\\
 	e^{'\lambda+1}&= cb
 \end{align}

\begin{rem}

The reader may notice that this description of the middle level $k^{\bigstar}_{C_2}(B_{C_4}\Sigma_{2+})$ is rather different from the one given in Proposition \ref{BiggestProposition}. Let us now explain this discrepancy. First, the relation
\begin{gather}
	\bar u_{\lambda}e^{au}=\sqrt{\bar a_{\lambda}\bar u_{\lambda}}\bar e^{u}
\end{gather} 
allows us to replace $e^{au}$ by the quotient $$\frac{\sqrt{\bar a_{\lambda}\bar u_{\lambda}}\bar e^{u}}{\bar u_\lambda}$$ which is why $e^{au}$ does not appear in Proposition \ref{BiggestProposition} but $(\sqrt{\bar a_{\lambda}\bar u_{\lambda}}\bar e^{u})/\bar u_\lambda$ does.\\
Second, in Lemma \ref{ResOfea}, we shall see that $\tilde e^{a}+\bar u_\sigma e^{'\lambda+1}$ is the restriction of a top level generator $e^{a}$, which we denote by $\bar e^{a}$. We can replace the generator $\tilde e^{a}$ by the element $\bar e^{a}$ and get the relation:
 \begin{gather}
 \bar u_\sigma \bar u_\lambda e^{'\lambda+1}=	\bar u_{\lambda}\bar e^{a}+\bar a_{\lambda}\bar e^{u}+\sqrt{\bar a_{\lambda}\bar u_{\lambda}}\bar u_{\sigma}\bar e^{\lambda}
 \end{gather} 
 Thus we can replace the generator $e^{'\lambda+1}$ by the quotient $$\frac{\bar u_\sigma^{-1}\bar a_{\lambda}\bar e^{u}+\sqrt{\bar a_{\lambda}\bar u_{\lambda}}\bar e^{\lambda}}{\bar u_\lambda}$$
 which is what we do in the description of the middle level $k^{\bigstar}_{C_2}(B_{C_4}\Sigma_{2+})$ found in Proposition \ref{BiggestProposition}. For our convenience, we shall continue to use the generators $\tilde e^{j,a}, e^{j,au}, e^{'j\rho+\lambda+1}$ in the following subsections instead of their replacements. 
 
\end{rem}

\subsection{Top level differentials}

In this subsection, we compute the top level of the $E_{\infty}$ page.
 
 
 From subsection \ref{d1comp}, we know that (the top level of) the $E_2$ page is generated by $$e^{j\rho}, \alpha e^{j\rho+\sigma}, \bar u_\sigma^{-i}\sqrt{\bar a_\lambda \bar u_\lambda}^{\epsilon}e^{j\rho+\lambda+\epsilon'}$$
 where $i,j\ge 0$, $\epsilon,\epsilon'=0,1$ and $\alpha\in k^{C_4}_{\bigstar}\setminus \{u_\sigma^m\text{ , }m\ge 0\}$. We also have the relation:
  \begin{equation}
 v\bar u_{\sigma}e^{j\rho+\lambda}=0
 \end{equation}
 
 For degree reasons, the elements $e^{j\rho}$ survive the spectral sequence. 
 
 The elements $\bar u_\sigma^{-i}e^{j\rho+\lambda+1}, \bar u_\sigma^{-i}\sqrt{\bar a_{\lambda}\bar u_{\lambda}}e^{j\rho+\lambda+1}$ are transfers hence also survive (by the middle level computation of subsection \ref{MidLevelComp}). 
 
 If $\alpha\in k^{C_4}_{\bigstar}$ is a transfer then so are the elements $\alpha e^{j\rho+\sigma}$ and thus they survive. The remaining elements $\alpha\in k^{\bigstar}_{C_4}\setminus \{u_\sigma^m\text{ , }m\ge 0\}$ can be broken into three categories:
 \begin{itemize}
 	\item multiples of $a_{\lambda}$
 	\item multiples of $u_{\lambda}/u_{\sigma}^i$
 	\item $a_{\sigma}u_{\sigma}^i$
 \end{itemize}

 \begin{prop}The elements $a_{\lambda}e^{j\rho+\sigma}$ survive the spectral sequence, while the elements $a_{\sigma}u_{\sigma}^ie^{j\rho+\sigma}$ support nontrivial differentials:
 	\begin{equation}
 	d^2(a_{\sigma}u_{\sigma}^ie^{j\rho+\sigma})=v\bar u_{\sigma}^{i+2}e^{j\rho+\lambda+1}
 	\end{equation}
 	for $i,j\ge 0$.
 \end{prop}
 
 \begin{proof}
 	The elements $a_{\lambda}e^{j\rho+\sigma}$ can only support $d^3(a_{\lambda}e^{j\rho+\sigma})=e^{(j+1)\rho}$ and applying restriction shows that this cannot happen. 
 	
 	Fix $j\ge 0$. For degree reasons, the only differential $a_{\sigma}e^{j\rho+\sigma}$ can support is $d^2(a_{\sigma}e^{j\rho+\sigma})=v\bar u_{\sigma}^2e^{j\rho+\lambda+1}$. 
 	If $a_\sigma e^{j\rho+\sigma}$ survives then it lifts to a unique element $\alpha$ of $k^{j\rho+2\sigma}_{C_4}$, while $a_\lambda e^{{j\rho+\sigma}}$ has two possible lifts to $k^{j\rho+2\sigma}_{C_4}$ that differ by $\Tr_2^4(u_\sigma e^{'j\rho+\lambda+1})$. Both lifts have the same restriction, which by Lemma \ref{ResOfea} is computed to be $\tilde e^{j,a}+\bar u_\sigma e^{'j\rho+\lambda+1}$ (the proof of the Lemma works regardless of the survival of $a_\sigma e^{j\rho+\sigma}$). Now one of those lifts, that we shall call $\beta$, satisfies:
 	$$\frac{a_\sigma^2}{a_\lambda}\beta=a_\sigma \alpha$$
 	in $k^{j\rho+3\sigma}_{C_4}(B_{C_4}\Sigma_{2+})$. Applying $\Res^4_2$ gives that
 	$$\Res^4_2\Big(\frac{a_\sigma^2}{a_\lambda}\Big)\Res^4_2(\beta)=0\implies v\bar u_\sigma^2(\tilde e^{j,a}+\bar u_\sigma e^{'j\rho+\lambda+1})=0\implies v\bar u_\sigma^3 e^{'j\rho+\lambda+1}=0$$
 	which contradicts the computation of the module structure of the middle level.
 \end{proof}

\begin{rem}\label{ProgramDifferential}The non-survival of $a_{\sigma}e^{\sigma}$ is consistent with the computation that $k^{2\sigma}_{C_4}$ has dimension $1$ (spanned by $a_{\sigma}^2e^0$) by the computer program of \cite{Geo19}.
\end{rem}

All the other elements of $E_2$ survive the spectral sequence:
 
 \begin{prop}The elements $(u_\lambda/u_\sigma^i)e^{j\rho+\sigma}, \bar u_\sigma^{-i}e^{j\rho+\lambda}, \bar u_\sigma^{-i}\sqrt{\bar a_{\lambda}\bar u_{\lambda}}e^{j\rho+\lambda}$ survive the spectral sequence for $i,j\ge 0$.
 	\end{prop}
 \begin{proof}We work page by page. On $E_2$ we have:
 	 	\begin{gather}
 	d^2(\bar u_\sigma^{-i}e^{j\rho+\lambda})=\epsilon_1\frac{\theta}{a_\sigma u_\sigma^{i-2}}e^{(j+1)\rho}\\
 	d^2(\bar u_\sigma^{-i}\sqrt{\bar a_{\lambda}\bar u_{\lambda}}e^{j\rho+\lambda})=\epsilon_2\frac{\theta}{a_\sigma^2 u_\sigma^{i-3}}a_\lambda e^{(j+1)\rho}+\epsilon_3\frac{u_\lambda}{u_\sigma^{i+1}}e^{(j+1)\rho}
 	\end{gather}
 	where $\epsilon_i=0,1$. Multiplying by $a_\sigma$ and using that $a_\sigma\bar u_\sigma^{-i}e^{j\rho+\lambda}=0$ and that $a_\sigma\bar u_\sigma^{-i}\sqrt{\bar a_{\lambda}\bar u_{\lambda}}e^{j\rho+\lambda}=0$ shows that $\epsilon_1=\epsilon_2=\epsilon_3=0$.
 	
 	On $E_3$ we have:
 	 	 	\begin{gather}
 	d^3(\bar u_\sigma^{-i}e^{j\rho+\lambda})=\epsilon_1\frac{\theta}{a_\sigma^2 u_\sigma^{i-2}}e^{(j+1)\rho+\sigma}\\
 	d^3(\bar u_\sigma^{-i}\sqrt{\bar a_{\lambda}\bar u_{\lambda}}e^{j\rho+\lambda})=\epsilon_2\frac{\theta}{a_\sigma^3 u_\sigma^{i-3}}a_\lambda e^{(j+1)\rho+\sigma}\\
 	d^3\Big(\frac{u_{\lambda}}{u_{\sigma}^i}e^{j\rho+\sigma}\Big)=\epsilon_3\frac{\theta}{a_{\sigma}^2u_{\sigma}^{i-4}}e^{(j+1)\rho}
 	\end{gather}
 	where again $\epsilon_i=0,1$. We see that $\epsilon_1=\epsilon_2=0$ by multiplication with $a_\sigma$, while $\epsilon_3=0$ can be seen by multiplying with $a_\sigma^2$.
 	
 	The pattern of higher differentials is the same as in $E_2,E_3$ and the same arguments show that there are no higher differentials.
 \end{proof}

In conclusion:

\begin{cor}\label{EinftyGen} The $E_{\infty}$ page is generated as a $k^{\bigstar}_{C_4}$ module by $$e^{j\rho}, a_{\lambda}e^{j\rho+\sigma},(u_{\lambda}/u_{\sigma}^i)e^{j\rho+\sigma},\bar u_\sigma^{-i}\sqrt{\bar a_{\lambda}\bar u_{\lambda}}^{\epsilon}e^{j\rho+\lambda+\epsilon'}$$ where $i,j\ge 0$ and $\epsilon,\epsilon'=0,1$. We have relations: $$v\bar u_{\sigma}e^{j\rho+\lambda}=v\bar u_{\sigma}^2e^{j\rho+\lambda+1}=0$$
	\end{cor}
 
 \subsection{Coherent lifts}
 If we have a top level element $\alpha\in E_\infty^{s,V}$ and $E_\infty^{t,V}=0$ for $t>s$ then $\alpha$ lifts uniquely to $k^{\bigstar}_{C_4}(B_{C_4}\Sigma_{2+})$. If on the other hand $E_\infty^{t,V}\neq 0$ for some $t>s$, then there are multiple choices of lifts of $\alpha$. 
 
 When it comes to fractions $y/x$, we should make sure our choices of lifts are "coherent". Let us explain what that means with an example: The element  $u_\lambda e^\sigma$ has a unique lift $x_0$ while $(u_\lambda/u_\sigma^i)e^\sigma$ has multiple distinct lifts if $i\ge 5$. If we choose $x_i$ to lift $(u_\lambda/u_\sigma^i)e^\sigma$ then it will always be true that $u_\sigma^ix_i=x_0$; however, we shouldn't write $x_i=x_0/u_\sigma^i$ unless we can also guarantee that:
 \begin{equation}
 u_\sigma x_i=x_{i-1}
 \end{equation}
 This expresses the coherence of fractions (also discussed in subsection \ref{ROC4HomolPt} and Appendix \ref{AppendixPoint}) which is the cancellation property:
 \begin{equation}
 u_\sigma \frac{u_\lambda}{u_\sigma^i}e^\sigma=\frac{u_\lambda}{u_\sigma^{i-1}}e^\sigma
 \end{equation}
 This holds on $E_\infty$ and we also want it to hold on $k^{\bigstar}_{C_4}(B_{C_4}\Sigma_{2+})$.
 
 One more property enjoyed by the $(u_\lambda/u_\sigma^i)e^\sigma$ is that $a_\sigma^2(u_\lambda/u_\sigma^i)e^\sigma=0$; it turns out that there are unique lifts $x_i$ of $(u_\lambda/u_\sigma^i)e^\sigma$ such that $a_\sigma^2x_i=0$ and those lifts also satisfy the coherence property $u_\sigma x_i=x_{i-1}$:

  \begin{prop}\label{HalfCoh}For $i,j\ge 0$, there are unique lifts $e^{j,u}/u_\sigma^i, e^{j\rho+\lambda}/u_\sigma^i$ of the elements $(u_\lambda/u_\sigma^i) e^{j\rho+\sigma}, \bar u_\sigma^{-i}e^{j\rho+\lambda}$ respectively that satisfy:
	\begin{gather}
		a_\sigma^2\frac{e^{j,u}}{u_\sigma^i}=0\\
		a_\sigma^2\frac{e^{j\rho+\lambda}}{u_\sigma^i}=0
	\end{gather}
These lifts are also coherent.
\end{prop}
\begin{proof}

Fix $i,j\ge 0$. We first deal with lifts of $(u_\lambda/u_\sigma^i) e^{j\rho+\sigma}$.

Existence: Fix $\bigstar$ to be the $RO(C_4)$ degree of $(u_\lambda/u_\sigma^i) e^{j\rho+\sigma}$ and write $F^s$ for the decreasing filtration on $k^{\bigstar}_{C_4}(B_{C_4}\Sigma_{2+})$ defining the spectral sequence, namely:
\begin{equation}
E_\infty^{s,\bigstar}=F^s/F^{s+1}
\end{equation}
We start with any random lift $\alpha_0\in F^{4j+1}$ of $(u_\lambda/u_\sigma^i) e^{j\rho+\sigma}$; if $a_\sigma^2\alpha_0=0$ then we are done. Otherwise take $s_0$ maximal with $a_\sigma^2\alpha_0\in F^{s_0}$; since $a_\sigma^2(u_\lambda/u_\sigma^i)e^{j\rho+\sigma}=0$ we have $s_0>4j+1$. In fact $s_0>4j+2$ since $E^{4j+2,\bigstar}=0$.\\
We now prove that $s_0>4j+3$: $E^{4j+3,\bigstar}$ is spanned by $\bar u_\sigma^{3-i}e^{j\rho+\lambda+1}$ so we need to investigate the possibility $a_\sigma^2\alpha=\bar u_\sigma^{3-i}e^{j\rho+\lambda+1}$ on $E^{4j+3,\bigstar}$. Multiplying by $u_\sigma^i$ reduces us to the case $i=0$, where $u_\sigma^i\alpha$ is the unique lift of $u_\lambda e^{j\rho+\sigma}$. But we can see directly that $(a_\sigma^2/a_\lambda)u_\sigma^i\alpha=0$ for degree reasons, hence $a_\sigma^2u_\sigma^i\alpha=0$ as well.\\
As $s_0>4j+3$, we can see directly that $F^{s_0}/F^{s_0+1}=E_\infty^{s_0,\bigstar}$ is generated by an element $\beta e^V$ where $\beta\in k^{\bigstar-V}_{C_4}$ is divisible by $a_\sigma^2$. If $\alpha'\in F^{s_0}$ is a lift of $(\beta/a_\sigma^2)e^V$ then $\alpha_1=\alpha_0+\alpha'$ is a lift of $(u_\lambda/u_\sigma^i) e^{j\rho+\sigma}$. If $a_\sigma^2\alpha_1=0$ then we are done, otherwise $a_\sigma^2\alpha_1\in F^{s_1}$ for $s_1>s_0$ so we get $\alpha_2$ by the same argument as above. Since $F^s=0$ for large enough $s$, this inductive process will eventually end with the desired lift.
	
Uniqueness: If $\alpha,\alpha'$ are two lifts of $(u_\lambda/u_\sigma^i) e^{j\rho+\sigma}$ then their difference is a finite sum $p$ of elements $\beta'e^V$ where each $\beta'\in  k^{\bigstar}_{C_4}$ is a fraction with $a_\sigma^2$ in its denominator. If $a_\sigma^2\alpha=a_\sigma^2\alpha'=0$ then $a_\sigma^2p=0\implies a_\sigma^2 \beta'=0\implies \beta'=0\implies p=0$.
	
Coherence: Unfix $i$ and let $x_i$ be the lift of $(u_\lambda/u_\sigma^i) e^{j\rho+\sigma}$ with $a_\sigma^2x_i=0$. Then $u_\sigma x_i$ is a lift of $(u_\lambda/u_\sigma^{i-1}) e^{j\rho+\sigma}$ and $a_\sigma^2(u_\sigma x_i)=0$ hence by uniqueness: $$u_\sigma x_i=x_{i-1}$$

The case of $\bar u_\sigma^{-i}e^{j\rho+\lambda}$ is near identical to what we did above for $(u_\lambda/u_\sigma^i) e^{j\rho+\sigma}$. The changes are as follows. First, $s_0>4j+2$ (instead of $s_0>4j+1$). Next we can see that $s_0>4j+4$ if $i>1$, and multiplying by $u_\sigma$ also proves the $i=0,1$ cases (this replaces the argument that showed $s_0>4j+3$). The rest of the arguments are identical.
\end{proof}

\subsection{Top level generators}\label{TopLevelGeneration}

The elements $e^{j\rho}$ have unique lifts to $k^{\bigstar}_{C_4}(B_{C_4}\Sigma_{2+})$ that we continue to denote by $e^{j\rho}$.

On the other hand, for each $j\ge 0$ there are two possible lifts of $a_{\lambda}e^{j\rho+\sigma}$. There is no good way to make a  unique choice at this point, so we shall write $e^{j,a}$ for either. 

In this subsection we shall prove: 
  \begin{prop}\label{Generation}The $k^{\bigstar}_{C_4}$ module $k^{\bigstar}_{C_4}(B_{C_4}\Sigma_{2+})$  is generated by
	\begin{equation}
	e^{j\rho}, e^{j,a},\frac{e^{j,u}}{u_{\sigma}^i},\frac{e^{j\rho+\lambda}}{u_{\sigma}^i}\end{equation}
	where $i,j\ge 0$.
\end{prop}

By Corollary \ref{EinftyGen} it suffices to prove that the $k^{\bigstar}_{C_4}$-algebra generated by the elements $e^{j\rho},e^{j,a},e^{j,u}/u_\sigma^i, e^{j\rho+\lambda}/u_\sigma^i$ contains lifts of the elements $$\bar u_\sigma^{-i}e^{j\rho+\lambda+1},\sqrt{\bar a_{\lambda}\bar u_{\lambda}}\bar u_\sigma^{-i}e^{j\rho+\lambda}, \sqrt{\bar a_{\lambda}\bar u_{\lambda}}\bar u_\sigma^{-i}e^{j\rho+\lambda+1}\in E_\infty$$

\begin{lem}\label{Liftoflambdaplus1}The elements 
	\begin{equation}
		\frac{e^{j\rho+\lambda+1}}{u_\sigma^i}:=\Tr_2^4(\bar u_\sigma^{-i}e^{'j\rho+\lambda+1})
	\end{equation}
are coherent lifts of $\bar u_\sigma^{-i}e^{j\rho+\lambda+1}\in E_\infty$. Furthermore,
	\begin{equation}
	a_\sigma \frac{e^{j\rho+\lambda}}{u_\sigma^i}= \frac{e^{j\rho+\lambda+1}}{u_\sigma^{i-1}}
\end{equation}
\end{lem}

\begin{proof}We see directly that $\Tr_2^4(\bar u_\sigma^{-i}e^{'j\rho+\lambda+1})$ lift $\bar u_\sigma^{-i}e^{j\rho+\lambda+1}$ and coherence follows from the Frobenius relations.
	
Next, we see directly from the $E_\infty$ page that $e^{j\rho+\lambda}$ is not in the image of the transfer $\Tr_2^4$. Since $Ker(a_\sigma)=Im(\Tr_2^4)$ in  $k^{\bigstar}_{C_4}(B_{C_4}\Sigma_{2+})$, we must have a module extension of the form:
\begin{equation}
	a_\sigma e^{j\rho+\lambda}=u_\sigma e^{j\rho+\lambda+1}
\end{equation}
By Proposition \ref{HalfCoh}, $a_\sigma^2 e^{j\rho+\lambda}/u_\sigma^i= 0$ hence $a_\sigma e^{j\rho+\lambda}/u_\sigma^i$ is a transfer. The equation above shows that $a_\sigma e^{j\rho+\lambda}/u_\sigma^i\neq 0$ and the only way $a_\sigma e^{j\rho+\lambda}/u_\sigma^i$ can be a nonzero transfer is for $a_\sigma e^{j\rho+\lambda}/u_\sigma^i=\Tr_2^4(\bar u_\sigma^{-i+1}e^{'j\rho+\lambda+1})$.
\end{proof}

Before we can lift the rest of the $E_\infty$ generators, we will need the following exotic restriction:
\begin{lem}\label{ResOfea}Both choices of $e^{j,a}$ have the same (exotic) restriction:
		\begin{equation}
	\Res^4_2(e^{j,a})=\tilde e^{j,a}+\bar u_{\sigma}e^{'j\rho+\lambda+1}
	\end{equation}
\end{lem}
\begin{proof}The two choices of $e^{j,a}$ differ by $u_\sigma e^{j\rho+\lambda+1}=\Tr_2^4(\bar u_\sigma e^{'j\rho+\lambda+1})$ hence have the same restriction. From the $E_\infty$ page:
	\begin{gather}
	\Res^4_2(e^{j,a})=\tilde e^{j,a}+\epsilon\bar u_{\sigma}e^{'j\rho+\lambda+1}
	\end{gather}
	where $\epsilon=0,1$. Transferring this gives
	\begin{gather}
	\Tr^4_2(\tilde e^{j,a})=\epsilon u_{\sigma}e^{j\rho+\lambda+1}
	\end{gather}
	Now transferring the middle level relation $$\bar u_{\lambda}\tilde e^{j,a}=\bar a_{\lambda}\bar e^{j,u}+\sqrt{\bar a_{\lambda}\bar u_{\lambda}}\bar u_{\sigma}\bar e^{j\rho+\lambda}$$ shows that $$u_{\lambda}\Tr_2^4(\tilde e^{j,a})=a_{\sigma}u_{\lambda}e^{j\rho+\lambda}
	$$ 
	and thus $\Tr_2^4(\tilde e^{j,a})\neq 0$ which proves $\epsilon=1$.\end{proof}

\begin{lem}\label{Liftoflambdasqrt}The elements 
	\begin{equation}
\frac{e^{j\rho+\lambda}_{\surd}}{u_\sigma^i}:=\frac{u_{\lambda}}{u_\sigma^{i+1}}e^{j,a}+a_{\lambda}\frac{e^{j,u}}{u_\sigma^{i+1}}
	\end{equation}
are coherent lifts of $\bar u_\sigma^{-i}\sqrt{\bar a_\lambda\bar u_\lambda}e^{j\rho+\lambda}\in E_\infty$.
\end{lem}

\begin{proof}Fix $i,j\ge 0$ and fix $\bigstar$ to be the degree of the element 	
	\begin{equation}\frac{u_{\lambda}}{u_\sigma^i}e^{j,a}+a_{\lambda}\frac{e^{j,u}}{u_\sigma^i}
	\end{equation}
	This element is by definition in filtration $4j+1$, however its projection to $E_\infty^{4j+1,\bigstar}$ is:
	\begin{equation}\frac{u_{\lambda}}{u_\sigma^i}a_\lambda e^{j\rho+\sigma}+a_{\lambda}\frac{u_\lambda }{u_\sigma^i}e^{j\rho+\sigma}=0
	\end{equation}
	so it is actually in filtration $4j+2$. But observe that $E_\infty^{4j+2,\bigstar}$ is generated by $\bar u_\sigma^{-i+1}\sqrt{\bar a_\lambda\bar u_\lambda}e^{j\rho+\lambda}$ so it suffices to check that
	\begin{equation}\frac{u_{\lambda}}{u_\sigma^i}e^{j,a}+a_{\lambda}\frac{e^{j,u}}{u_\sigma^i}
	\end{equation}
	is not $0$ when projected to $E_\infty^{4j+2,\bigstar}$. Multiplying by $u_\sigma^i$ reduces us to the case $i=0$ and then:
		\begin{equation}\Res^4_2(u_{\lambda}e^{j,a}+a_{\lambda}e^{j,u})=\bar u_\lambda\tilde e^{j,a}+\bar u_{\sigma}\bar u_\lambda e^{'j\rho+\lambda+1}+\bar a_\lambda \bar e^{j,u}=\bar u_\sigma\sqrt{\bar a_\lambda\bar u_\lambda}\bar e^{j\rho+\lambda}+\bar u_{\sigma}\bar u_\lambda e^{'j\rho+\lambda+1}
	\end{equation}
	using Lemma \ref{ResOfea} and the middle level computation of subsection \ref{MidLevelComp}. Projecting this restriction to $E_\infty^{4j+2,\bigstar}$ returns
			\begin{equation}\bar u_\sigma\sqrt{\bar a_\lambda\bar u_\lambda}\bar e^{j\rho+\lambda}\neq 0
	\end{equation}
	as desired.
	
	Coherence of $e^{j\rho+\lambda}_{\surd}/u_\sigma^i$ follows from the coherence of $u_\lambda/u_\sigma^i$ and $e^{j,u}/u_\sigma^i$.

\end{proof}

\begin{lem}\label{Liftoflambdaplus1sqrt}The elements 
	\begin{equation}
		\frac{e^{j\rho+\lambda+1}_\surd}{u_\sigma^i}:=\Tr_2^4(\bar u_\sigma^{-i}\sqrt{\bar a_\lambda\bar u_\lambda}e^{'j\rho+\lambda+1})
	\end{equation}
	are coherent lifts of $\bar u_\sigma^{-i}\sqrt{\bar a_\lambda\bar u_\lambda}e^{j\rho+\lambda+1}\in E_\infty$. Furthermore,
	\begin{equation}
		a_\sigma \frac{e^{j\rho+\lambda}_\surd}{u_\sigma^i}= \frac{e^{j\rho+\lambda+1}}{u_\sigma^{i-1}}
	\end{equation}
\end{lem}
 
 \begin{proof}The fact that these transfers are lifts follows from the $E_\infty$ page; coherence follows from the Frobenius relations. We check the equality directly:
     \begin{align}
 a_\sigma\frac{e^{j\rho+\lambda}_{\surd}}{u_\sigma^i}&=\frac{a_\sigma u_{\lambda}}{u_\sigma}e^{j,a}+a_{\lambda}\frac{a_\sigma e^{j,u}}{u_\sigma}=\Tr_2^4(\bar u_\sigma^{-i}\sqrt{\bar a_\lambda\bar u_\lambda})e^{j,a}+a_{\lambda}\Tr_2^4(e^{j,au}\bar u_{\sigma}^{-i})=\\
 &=\Tr_2^4(\bar u_\sigma^{-i}\sqrt{\bar a_\lambda\bar u_\lambda}\tilde e^{j,a}+\sqrt{\bar a_\lambda\bar u_\lambda}\bar u_\sigma^{-i+1} e^{'j\rho+\lambda+1}+\bar a_\lambda e^{j,au}\bar u_\sigma^{-i})=\\
 &=\Tr_2^4(\bar a_\lambda \bar u_\sigma^{-i+1} \bar e^{j\rho+\lambda}+\sqrt{\bar a_\lambda\bar u_\lambda}\bar u_\sigma^{-i+1} e^{'j\rho+\lambda+1})=\\
 &=\Tr_2^4(\sqrt{\bar a_\lambda\bar u_\lambda}\bar u_\sigma^{-i+1} e^{'j\rho+\lambda+1})=\\
 &=\frac{e^{j\rho+\lambda+1}_\surd}{u_\sigma^{i-1}}
 \end{align}
 We used the middle level relation $	\bar a_{\lambda}e^{j,au}=\sqrt{\bar a_{\lambda}\bar u_{\lambda}}\tilde e^{j,a}+\bar a_{\lambda}\bar u_{\sigma}\bar e^{j\rho+\lambda}$ and the fact that $\bar u_\sigma^{-i} \bar e^{j\rho+\lambda}$ is the restriction of $e^{j\rho+\lambda}/u_\sigma^{i}$ which follows from the same fact on $E_\infty$.
  \end{proof}

 Lemmas \ref{Liftoflambdaplus1}, \ref{Liftoflambdasqrt}, \ref{Liftoflambdaplus1sqrt} combined with Corollary \ref{EinftyGen} prove Proposition \ref{Generation}.

\subsection{Mackey functor structure} 

 \begin{prop}The Mackey functor structure of $k^{\bigstar}(B_{C_4}\Sigma_{2+})$ is determined by: 
 	\begin{gather}
 	\Res^4_2(e^{j\rho})=\bar e^{j\rho}\text{ , } 	\Res^4_2\Big(\frac{e^{j,u}}{u_\sigma^i}\Big)=\bar e^{j,u}\bar u_{\sigma}^{-i}\text{ , }	\Res^4_2\Big(\frac{e^{j\rho+\lambda}}{u_\sigma^i}\Big)=\bar e^{j\rho+\lambda}\bar u_{\sigma}^{-i}\\
 	 \Tr^4_2(e^{j,au}\bar u_{\sigma}^{-i})=a_{\sigma}\frac{e^{j,u}}{u_{\sigma}^{i+1}}\text{ , } 	\Tr^4_2\Big(e^{'j\rho+\lambda+1}\bar u_{\sigma}^{-i}\Big)=a_\sigma\frac{e^{j\rho+\lambda}}{u_\sigma^{i+1}}\\
 	 \Res^4_2(e^{j,a})=\tilde e^{j,a}+\bar u_{\sigma}e^{'j\rho+\lambda+1}
 	\end{gather}
where $i,j\ge 0$.
 \end{prop}
 
 \begin{proof}We can see directly that there are no Mackey functor extensions for $e^{j\rho}, e^{j,u}/u_{\sigma}^i$ and $e^{j\rho+\lambda}/u_{\sigma}^i$. The rest were established in the previous two subsections, apart from:
 	 	\begin{gather}
 		\Tr^4_2(e^{j,au}\bar u_{\sigma}^{-i})=a_{\sigma}\frac{e^{j,u}}{u_{\sigma}^{i+1}}
 	\end{gather}
 To see this, recall that $a_{\sigma}^2(e^{j,u}/u_{\sigma}^i)=0$ hence $a_{\sigma}(e^{j,u}/u_{\sigma}^i)$ is a transfer. Moreover, $a_{\sigma}(e^{j,u}/u_{\sigma}^i)\neq 0$ which is seen on the $E_\infty$ page, and the only way that $a_{\sigma}(e^{j,u}/u_{\sigma}^i)$ can be a nonzero transfer is for $a_{\sigma}(e^{j,u}/u_{\sigma}^i)=\Tr_2^4(e^{j,au}\bar u_{\sigma}^{-i})$.
 \end{proof}

\subsection{Top level module relations}
With the exception of relations expressing coherence ($u_\sigma (e^{j,u}/u_\sigma^i)=e^{j,u}/u_\sigma^{i-1}$ and $u_\sigma (e^{j\rho+\lambda}/u_\sigma^i)=e^{j\rho+\lambda}/u_\sigma^{i-1}$), the rest of the module relations are:
  \begin{prop}\label{Generation+Rel}The $k^{\bigstar}_{C_4}$ module $k^{\bigstar}_{C_4}(B_{C_4}\Sigma_{2+})$ is generated by:
	\begin{equation}
	e^{j\rho}, e^{j,a},\frac{e^{j,u}}{u_{\sigma}^i},\frac{e^{j\rho+\lambda}}{u_{\sigma}^i}\end{equation}
	under the relations:
	\begin{gather}
	\frac{a_{\sigma}^2}{a_{\lambda}^m}\frac{e^{j,u}}{u_{\sigma}^i}=0\\
	\frac{\frac{\theta}{a_{\lambda}}a_{\sigma}}{u_{\sigma}^{i-2}a_{\lambda}^{m-1}}e^{j,a}+\frac{s}{u_{\sigma}^{i-1}a_{\lambda}^{m-2}}e^{j,u}=\frac{a_{\sigma}^2}{a_{\lambda}^m}\frac{e^{j\rho+\lambda}}{u_{\sigma}^i}
	\end{gather}
	for $i,j,m\ge 0$. 
\end{prop}
 \begin{proof}For $m>0$, we have the possible extensions:
 	\begin{equation}
 	\frac{a_{\sigma}^2}{a_{\lambda}^m}\frac{e^{j,u}}{u_{\sigma}^i}=\sum_*\epsilon_{*}\frac{\theta}{a_\sigma^*u_\sigma^*a_\lambda^*}e^{*\rho+*}
 	\end{equation}
 	where each $*$ denotes a nonnegative index (with different instances of $*$ being possibly different indices) and each $\epsilon_*=0,1$. Thus, multiplication by $a_\lambda$ is an isomorphism for both sides which reduces us to $m=1$. For $m=1$ and $i>0$ there are no extensions i.e. $\epsilon_*=0$ for all $*$. This establishes 
 	\begin{equation}
 		\frac{a_{\sigma}^2}{a_{\lambda}^m}\frac{e^{j,u}}{u_{\sigma}^i}=0
 	\end{equation}
	Similarly, if $m>0$, we have the possible extensions:
	\begin{equation}
	\frac{a_{\sigma}^2}{a_{\lambda}^m}\frac{e^{j\rho+\lambda}}{u_{\sigma}^i}=\frac{s}{u_{\sigma}^{i-2}a_{\lambda}^{m-1}}e^{j\rho+\lambda}_{\surd}+\sum_*\epsilon_*\frac{\theta}{a_\sigma^*u_\sigma^*a_\lambda^*}e^{*\rho+*}
	\end{equation}
	and multiplying with $a_{\lambda}^m$ reduces us to:
		\begin{equation}
	a_{\sigma}^2\frac{e^{j\rho+\lambda}}{u_{\sigma}^i}=\sum_*\epsilon_*\frac{\theta}{a_\sigma^*u_\sigma^*a_\lambda^{*-m}}e^{*\rho+*}
	\end{equation}
	But
			\begin{equation}
	a_{\sigma}^2\frac{e^{j\rho+\lambda}}{u_{\sigma}^i}=a_\sigma\frac{e^{j\rho+\lambda+1}}{u_{\sigma}^{i-1}}=a_\sigma \Tr_2^4(e^{'j\rho+\lambda+1}\bar u_\sigma^{-i+1})=0
	\end{equation}
	hence $\epsilon_{*}=0$ for all $*$. Thus
	\begin{equation}
\frac{a_{\sigma}^2}{a_{\lambda}^m}\frac{e^{j\rho+\lambda}}{u_{\sigma}^i}=\frac{s}{u_{\sigma}^{i-2}a_{\lambda}^{m-1}}e^{j\rho+\lambda}_{\surd}
\end{equation}
and substituting
    \begin{gather}
e^{j\rho+\lambda}_{\surd}=\frac{u_{\lambda}}{u_\sigma}e^{j,a}+a_{\lambda}\frac{e^{j,u}}{u_\sigma}
\end{gather}
gives the desired relation. For $i=m=1$ we get $(a_\sigma^2/a_\lambda)(e^{j\rho+\lambda+1}/u_\sigma)=0$ which lifts the $E_\infty$ relations $v\bar u_\sigma e^{j\rho+\lambda}=v\bar u_\sigma^2e^{j\rho+\lambda+1}=0$.
\end{proof}
As special cases we get the relations:
	\begin{gather}a_{\sigma}^2\frac{e^{j\rho+\lambda}}{u_\sigma^i}=0\\
	\frac{a_{\sigma}^2}{a_\lambda}\frac{e^{j\rho+\lambda}}{u_\sigma}=0\\
	\frac{\frac{\theta}{a_{\lambda}}a_{\sigma}}{u_{\sigma}^i}e^{j,a}=\frac{a_{\sigma}^2}{a_{\lambda}}\frac{e^{j\rho+\lambda}}{u_{\sigma}^{i+2}}\\
	\frac{s}{a_{\lambda}^m}e^{j,u}=\frac{a_{\sigma}^2}{a_{\lambda}^{m+2}}\frac{e^{j\rho+\lambda}}{u_{\sigma}}\\
\end{gather}
for $i,j,m\ge 0$.

\subsection{Top level cup products}
 \begin{prop}As a $k^{\bigstar}_{C_4}$ algebra, $k^{\bigstar}_{C_4}(B_{C_4}\Sigma_{2+})$ is generated by $e^a,e^u/u_\sigma^i, e^\lambda/u_\sigma^i, e^\rho$.
 \end{prop}
\begin{proof}First of all, $e^{j\rho}=(e^{\rho})^j$ since there are no extensions in degree $j\rho$ (to see that $(e^{\rho})^j\neq 0$ apply restriction). Let $A$ be the algebra span of $e^a,e^u/u_\sigma^i, e^\lambda/u_\sigma^i, e^\rho$. To see that $e^{j,a}\in A$ observe:
$$e^{j\rho}e^a=\epsilon a_\sigma a_\lambda e^{j\rho}+e^{j,a}+\epsilon'u_\sigma e^{j\rho+\lambda+1}$$ 
and since $$e^{j\rho+\lambda+1}=\Tr_2^4(e^{'j\rho+\lambda+1})=\Tr_2^4(e^{j\rho}e^{'\lambda+1})=e^{j\rho}e^{\lambda+1}$$ 
we get that $e^{j,a}\in A$ regardless of the status of $\epsilon,\epsilon'$. 

Now suppose by induction that all elements in filtration $\le 4j$ are in $A$. We have that:
$$e^{j\rho}\frac{e^u}{u_\sigma^i}=\cdots+\frac{e^{j,u}}{u_\sigma^i}+\sum \epsilon_{*}\frac{\theta}{a_\sigma^*u_\sigma^*a_\lambda^*}e^{*\rho}+\sum \epsilon'_{*}\frac{\theta}{a_\sigma^*u_\sigma^*a_\lambda^*}e^{*,a}$$
where $\cdots$ are in filtration $<4j+1$ hence in $A$. Since $e^{*\rho}, e^{*,a}\in A$ for any $*\ge 0$, we get $e^{j,u}/u_\sigma^i\in A$. This establishes that everything in filtration $\le 4j+1$ is in $A$.

Finally,
$$e^{j\rho}\frac{e^\lambda}{u_\sigma^i}=\cdots+\frac{e^{j\rho+\lambda}}{u_\sigma^i}+\sum \epsilon_{*}\frac{\theta}{a_\sigma^*u_\sigma^*a_\lambda^*}e^{*\rho}+\sum \epsilon_{*}'\frac{\theta}{a_\sigma^*u_\sigma^*a_\lambda^*}e^{*,a}$$
where $\cdots$ are in filtration $<4j+2$, so by the same argument $e^{j\rho+\lambda}/u_\sigma^i\in A$ as well. This completes the induction step.\end{proof}
 
 Inverting $u_{\sigma},u_{\lambda}$ gives
 \begin{equation}
 k^{hC_4 \bigstar}[e^{\rho},e^a,e^u,e^{\lambda}]
 \end{equation}
 modulo relations,  which is isomorphic to
 \begin{equation}
 k^{hC_4\bigstar}(B_{C_4}\Sigma_{2+})=k[a_{\sigma},a_{\lambda},u_{\sigma}^{\pm},u_{\lambda}^{\pm},w]/a_{\sigma}^2\text{ , }|w|=1
 \end{equation}
There are two possible choices for $w$, differing by $a_\sigma u_\sigma^{-1}$, but both work equally well for the following arguments.

\begin{prop}\label{ModifyLocalization}After potentially replacing the generators $e^a,e^u/u_\sigma^i,e^\lambda/u_\sigma^i$ with algebra generators in the same degrees of $k^{\bigstar}_{C_4}(B_{C_4}\Sigma_{2+})$ and satisfying the same already established relations, the localization map
	\begin{equation}
		k^{\bigstar}_{C_4}(B_{C_4}\Sigma_{2+})\to  k^{hC_4\bigstar}(B_{C_4}\Sigma_{2+})
	\end{equation}
	 is given by:
	\begin{align}
e^u&\mapsto u_{\sigma}u_{\lambda}w\\
e^{\lambda}&\mapsto u_{\lambda}w^2\\
e^a&\mapsto u_{\sigma}u_{\lambda}w^3+u_{\sigma}a_{\lambda}w\\
e^{\rho}&\mapsto u_{\sigma}u_{\lambda}w^4+a_\sigma u_\lambda w^3+u_{\sigma}a_{\lambda}w^2+a_\sigma a_\lambda w
	\end{align}
\end{prop}
\begin{proof}
Using the $C_2$ result (see subsection \ref{BC2S2}), we have the correspondence on the middle level generators:
 \begin{itemize}
	\item $\bar e^u\mapsto \bar u_{\sigma}\bar u_{\lambda}w$
	\item $e^{\lambda}\mapsto \bar u_{\lambda}w^2$
	\item $\Res^4_2(e^a)\mapsto \bar u_{\sigma}(\bar u_{\lambda}w^3+\bar a_{\lambda}w)$
	\item $\bar e^{\rho}\mapsto \bar u_{\sigma}(\bar a_\lambda w^2+\bar u_\lambda w^4)$
\end{itemize}
from which we can deduce that the correspondence on top level is:
 \begin{itemize}
	\item $e^u\mapsto u_{\sigma}u_{\lambda}w+\epsilon_1a_{\sigma}u_{\lambda}$
	\item $e^{\lambda}\mapsto u_{\lambda}w^2+\epsilon_2a_{\sigma}u_{\sigma}^{-1}u_{\lambda}w$
	\item $e^a\mapsto u_{\sigma}u_{\lambda}w^3+\epsilon_3a_{\sigma}u_{\lambda}w^2+u_{\sigma}a_{\lambda}w$
	\item $e^{\rho}\mapsto u_{\sigma}u_{\lambda}w^4+\epsilon_4a_{\sigma}u_{\lambda}w^3+u_{\sigma}a_{\lambda}w^2+\epsilon_5a_{\sigma}a_{\lambda}w$
\end{itemize}
where the $\epsilon_i$ range in $0,1$.

We may add $\epsilon_1a_{\sigma}u_{\lambda}/u_{\sigma}^i$ to  $e^u/u_{\sigma}^i$ to force $\epsilon_1=0$; we may add $\epsilon_2a_{\sigma}e^u/u_{\sigma}^{i+2}$ to $e^{\lambda}/u_{\sigma}^i$ to force $\epsilon_2=0$ and we may add $\epsilon_3a_{\sigma}e^{\lambda}$ to $e^a$ to force $\epsilon_3=0$. \\
It remains to prove that $\epsilon_4=\epsilon_5=1$. This is a computation based on  the Bockstein homomorphism $\beta:k^{\bigstar}_{C_4}(X)\to k^{\bigstar+1}_{C_4}(X)$. For $X=S^0$ we have:
\begin{gather}
	\beta(a_\sigma)=\beta(a_\lambda)=\beta(u_\lambda)=0\\
	\beta(u_\sigma)=a_\sigma
\end{gather} 
For $X=B_{C_4}\Sigma_{2+}$ we see that $\beta(e^{\rho})=0$ for degree reasons ($k^{\rho+1}_{C_4}(B_{C_4}\Sigma_{2+})=0$) and in the homotopy fixed points, $\beta(w)=w^2, \beta(w^3)=w^4$. Thus applying $\beta$ on $e^{\rho}\mapsto u_{\sigma}u_{\lambda}w^4+\epsilon_4a_{\sigma}u_{\lambda}w^3+u_{\sigma}a_{\lambda}w^2+\epsilon_5a_{\sigma}a_{\lambda}w$ shows that $\epsilon_4=\epsilon_5=1$.
\end{proof}

\begin{prop}

 We have the multiplicative relations in $k^{\bigstar}_{C_4}(B_{C_4}\Sigma_{2+})$:
   \begin{gather}
    \frac{e^u}{u_{\sigma}^i}\frac{e^u}{u_{\sigma}^j}
    =\frac{u_{\lambda}}{u_{\sigma}^{i+j-2}}e^{\lambda}\\
    \frac{e^{\lambda}}{u_{\sigma}^i}\frac{e^u}{u_{\sigma}^j}
    =\frac{u_{\lambda}}{u_{\sigma}^{i+j}}e^a+a_{\lambda}\frac{e^u}{u_{\sigma}^{i+j}}\\
    e^a\frac{e^u}{u_{\sigma}^i}
    =\frac{u_{\lambda}}{u_{\sigma}^{i-1}}e^{\rho}+a_\sigma \frac{u_{\lambda}}{u_{\sigma}^i}e^a\\
    \frac{e^{\lambda}}{u_{\sigma}^i}\frac{e^{\lambda}}{u_{\sigma}^j}
    =\frac{u_{\lambda}}{u_{\sigma}^{i+j+1}}e^{\rho}+a_\sigma\frac{u_{\lambda}}{u_{\sigma}^{i+j+2}} e^a+a_{\lambda}\frac{e^{\lambda}}{u_{\sigma}^{i+j}}\\
     e^a\frac{e^{\lambda}}{u_{\sigma}^i}
     =\frac{e^u}{u_{\sigma}^{i+1}}e^{\rho}+a_\sigma \frac{u_\lambda}{u_{\sigma}^{i+1}}e^{\rho} \\
      (e^a)^2=
      u_{\sigma}e^{\lambda}e^{\rho}+a_\sigma \frac{e^u}{u_\sigma}e^\rho +u_{\sigma}a_{\lambda}e^{\rho}+a_\sigma a_\lambda e^a
 \end{gather}
 \end{prop}
 \begin{proof}First,
 	\begin{equation}
 	 \frac{e^u}{u_\sigma^i}\frac{e^u}{u_{\sigma}^j}=\epsilon_0a_{\lambda}\frac{u_{\lambda}}{u_{\sigma}^{i+j-2}}+\epsilon_1a_{\sigma}\frac{u_{\lambda}}{u_{\sigma}^{i+j}}e^u+\epsilon_2\frac{u_{\lambda}}{u_{\sigma}^{i+j-2}}e^{\lambda}+\cdots
 	\end{equation}
 	where $\epsilon_i=0,1$ and $\cdots$ is the sum of elements mapping to $0$ in homotopy fixed points, but all having denominator $a_{\sigma}^2$. Mapping to homotopy fixed points shows $\epsilon_0=\epsilon_1=0$ and $\epsilon_2=1$, while multiplying by $a_{\sigma}^2$ trivializes the LHS (by $a_{\sigma}^2(e^u/u_{\sigma}^i)=0$) and thus shows that $\cdots=0$.
 	
 	The same argument applied to:
 	 	\begin{equation}
 	\frac{e^\lambda}{u_\sigma^i}\frac{e^u}{u_{\sigma}^j}=\epsilon_0\frac{\theta a_{\lambda}^2}{a_\sigma u_{\sigma}^{i+j-2}}+\epsilon_1a_\sigma a_{\lambda}\frac{u_{\lambda}}{u_{\sigma}^{i+j}}+\epsilon_2a_{\lambda}\frac{e^u}{u_{\sigma}^{i+j}}+\epsilon_3\frac{u_{\lambda}}{u_{\sigma}^{i+j}}e^a+\epsilon_4a_\sigma\frac{u_{\lambda}}{u_{\sigma}^{i+j}}e^{\lambda}+\cdots
 	\end{equation}
 	shows 
 	 	 	\begin{equation}
 	\frac{e^\lambda}{u_\sigma^i}\frac{e^u}{u_{\sigma}^j}=\epsilon_0\frac{\theta a_{\lambda}^2}{a_\sigma u_{\sigma}^{i+j-2}}+a_{\lambda}\frac{e^u}{u_{\sigma}^{i+j}}+\frac{u_{\lambda}}{u_{\sigma}^{i+j}}e^a
 	\end{equation}
 	There are two ways to show that $\epsilon_0=0$: The first is to multiply with $a_{\sigma}u_\sigma^{i+j-2}$ and compute $a_\sigma e^\lambda (e^u/u_\sigma^2)$ using $a_\sigma e^\lambda=\Tr_2^4(e^{'\lambda+1})$ together with the Frobenius relation and our knowledge of the multiplicative structure of the middle level from subsection \ref{MidLevelComp}. The alternative is to observe that in the spectral sequence, if $a,b$ live in filtrations $\ge n$ then so does $ab$. Before the modifications to the generators done in the proof of Proposition \ref{ModifyLocalization}, $e^u/u_\sigma^i, e^a$ were in filtration $\ge 1$ and $e^\lambda/u_\sigma^i$ were in filtration $\ge 2$. Thus, with the original  generators, the extension for $e^\lambda e^u/u_\sigma^2$ does not involve the filtration $0$ term $a_{\lambda}^2\theta/a_\sigma $. This is true even after performing the modifications prescribed in the proof of Proposition \ref{ModifyLocalization}, since said modifications never involve terms with $\theta$. Thus $\epsilon_0=0$.

 Similarly we have:
 	 \begin{equation}
 	 e^a\frac{e^u}{u_{\sigma}^i}=\epsilon_0\frac{\theta a_{\lambda}^2}{u_{\sigma}^{i-4}}+\epsilon_1\frac{\theta a_{\lambda}}{a_\sigma u_{\sigma}^{i-3}}e^a+\epsilon_2a_\sigma a_{\lambda}\frac{e^u}{u_{\sigma}^i}+\epsilon_3a_\sigma\frac{u_{\lambda}}{u_{\sigma}^{i}}e^a+\epsilon_4a_\lambda\frac{e^{\lambda}}{u_{\sigma}^{i-2}}+
 	 \epsilon_5\frac{u_\lambda}{u_\sigma^{i-1}}e^\rho+\cdots
 	 \end{equation}
 for $i\ge 3$, and mapping to homotopy fixed points and multiplying by $a_\sigma^2$ shows
 	  	 \begin{equation}
 	 e^a\frac{e^u}{u_{\sigma}^i}=\epsilon_0\frac{\theta a_{\lambda}^2}{u_{\sigma}^{i-4}}+\epsilon_1\frac{\theta a_{\lambda}}{a_\sigma u_{\sigma}^{i-3}}e^a+a_\sigma\frac{u_{\lambda}}{u_{\sigma}^{i}}e^a+\frac{u_\lambda}{u_\sigma^{i-1}}e^\rho
 	 \end{equation}
 Multiplying by $a_\sigma$ and using that $a_{\sigma}(e^u/u_{\sigma}^i)=\Tr_2^4(e^{au}\bar u_{\sigma}^{-i})$ shows that $\epsilon_1=0$. To show $\epsilon_0=0$ we use the filtration argument above.
 	
 	These arguments also work with:
 		 \begin{align}
 	\frac{e^\lambda}{u_\sigma^i}\frac{e^\lambda}{u_{\sigma}^j}&=\epsilon_0\frac{\theta a_{\lambda}^2}{u_{\sigma}^{i+j-2}}+\epsilon_1\frac{\theta a_{\lambda}}{a_\sigma u_{\sigma}^{i+j-2}}e^a+\epsilon_2a_\sigma a_{\lambda}\frac{e^u}{u_{\sigma}^{i+j-2}}+\epsilon_3a_\sigma\frac{u_{\lambda}}{u_{\sigma}^{i+j+2}}e^a+\epsilon_4a_\lambda\frac{e^{\lambda}}{u_{\sigma}^{i+j}}+\\
 	&+\epsilon_5\frac{u_\lambda}{u_\sigma^{i+j+1}}e^\rho+\cdots\\
 	 e^a\frac{e^\lambda}{u_\sigma^i}&=\epsilon_6\frac{\theta a_{\lambda}}{a_\sigma u_{\sigma}^{i-2}}e^a+\epsilon_7a_\sigma a_{\lambda}\frac{e^\lambda}{u_{\sigma}^{i-1}}+\epsilon_8\frac{a_\sigma u_\lambda}{u_\sigma^{i+1}}e^\rho+\epsilon_9\frac{\theta a}{a_\sigma a_\lambda u_\sigma^{i-3}}e^\rho+\epsilon_{10}\frac{u_\lambda}{u_\sigma^{i+1}}e^\rho e^u+\cdots\\
 	 (e^a)^2&=\epsilon_{11}a_\sigma^2a_\lambda^2+\epsilon_{12}a_\sigma a_\lambda e^a+\epsilon_{13}u_\sigma a_\lambda e^\rho+\epsilon_{14}a_\sigma \frac{e^u}{u_\sigma}e^\rho+\epsilon_{16}u_\sigma e^\lambda e^\rho
 	\end{align}
 to complete the proof.
 \end{proof}

 We also have the nontrivial Bockstein:
 \begin{equation}
 	\beta(e^u/u_\sigma)=e^\lambda
 \end{equation}
 
\appendix

 \section{Pictures of the spectral sequence}\label{AppendixSS}
  
 In this appendix, we have included pictures of the $E_1$ page of the spectral sequence from section  \ref{BC4S2ss}. In each page, the three levels of the spectral sequence are drawn in three separate figures from top to bottom, using $(V,s)$ coordinates. For notational simplicity and due to limited space, we suppress the $e^V$'s and $x,gx$'s from the generators. The $e^V$'s can be recovered by looking at the filtration $s$ (e.g. in filtration $s=4j$ we get $e^{j\rho}$) and to denote the presence of $2$-dimensional vector spaces $k\{x,gx\}$ we write $k^2$ next to each generator. 
 
 For example, in the very first picture there is an element $x_{0,1}/u_{\sigma}^2$ in coordinates $(5,5)$. This represents the fact that the top level of $E_1^{5,5}$ is generated by $(x_{0,1}/u_{\sigma}^2)e^{\rho+\sigma}$. In the picture directly below and in the same coordinates we have $v \bar u_{\sigma}^{-2}$ meaning that the middle level of $E_1^{5,5}$ is generated by $(v\bar u_{\sigma}^{-2})\bar e^{\rho+\sigma}$. We have $$\Tr_2^4(v\bar u_{\sigma}^{-2}\bar e^{\rho+\sigma})=\frac{x_{0,1}}{u_{\sigma}^2}e^{\rho+\sigma}$$
 
 In the same picture, if we look at coordinates $(2,2)$ we see $v, vk^2, \overline{\bar u}_\lambda^{-1}k^2$ in the top, middle and bottom levels respectively. This represents that the three levels of $E_1^{2,2}$ are generated by $ve^{\lambda}(x+gx)$ for the top,  $v\bar e^{\lambda}x, v\bar e^{\lambda}gx$ for the middle, and $\overline{\bar u}_\lambda^{-1}\bar e^{\lambda}x, \overline{\bar u}_\lambda^{-1}\overline{\bar e}^{\lambda}gx$ for the bottom level. We have $$\Tr_2^4(v\bar e^{\lambda}x)=ve^{\lambda}(x+gx)$$
 These pictures are all  obtained automatically by the computer program of \cite{Geo19} available \href{https://github.com/NickG-Math/Mackey}{here}.
 \newpage

 \thispagestyle{empty}
 \begin{figure}[H]
 	\begin{center}
 		\begin{tikzpicture}[every node/.style={},xscale=1.5,yscale=1]
 		\foreach \from in {0,...,6}
 		\node at (\from,-0.8){$\from$};
 		\foreach \from in {0,...,6}
 		\node at (-0.6,\from){$\from$};
 		\draw [thick] (-0.3,-0.5) -- (6.6,-0.5);
 		\draw [thick] (-0.3,-0.5) -- (-0.3,6.5);
 		\node (1) at (0,0){$1$};
 		\node (0) at (1,1) {$0$};
 		\node at (2,2) {$v$};
 		\node at (3,3) {$v$};
 		\node at (3,4) {$\frac{x_{0,1}}{a_{\sigma}}$};
 		\node at (4,4)   {$\frac{x_{0,1}}{u_{\sigma}}$};
 		\node at (3,5) {$\frac{\theta}{a_{\lambda}}$};
 		\node at (4,5) {$\frac{x_{0,1}}{a_{\sigma}u_{\sigma}}$};
 		\node at (5,5) {$\frac{x_{0,1}}{u_{\sigma}^2}$};
 		\node at (4,6) {$\frac{v}{\bar a_{\lambda}}\bar u_{\sigma}^{-1}$};
 		\node at (5,6) {$\bar s \bar u_{\sigma}^{-1}$};
 		\node at (6,6) {$\frac{v}{\bar u_{\lambda}}\bar u_{\sigma}^{-1}$};
 		\end{tikzpicture}
 		\begin{tikzpicture}[every node/.style={},xscale=1.5,yscale=1]
 		\foreach \from in {0,...,6}
 		\node at (\from,-0.8){$\from$};
 		\foreach \from in {0,...,6}
 		\node at (-0.6,\from){$\from$};
 		\draw [thick] (-0.3,-0.5) -- (6.6,-0.5);
 		\draw [thick] (-0.3,-0.5) -- (-0.3,6.5);
 		\node (1) at (0,0){$1$};
 		\node (us) at (1,1) {$\bar u_{\sigma}^{-1}$};
 		\node (v) at (2,2) {$vk^2$};
 		\node (v') at (3,3) {$vk^2$};
 		\node at (4,4)   {$v\bar u_{\sigma}^{-1}$};
 		\node at (5,5) {$v\bar u_{\sigma}^{-2}$};
 		\node at (4,6) {$\frac{v}{\bar a_{\lambda}}\bar u_{\sigma}^{-1}k^2$};
 		\node at (5,6) {$\bar s \bar u_{\sigma}^{-1}k^2$};
 		\node at (6,6) {$\frac{v}{\bar u_{\lambda}}\bar u_{\sigma}^{-1}k^2$};
 		\draw [->] (us) -- (v);
 		\draw [->] (v) -- (v');
 		\end{tikzpicture}
 		\begin{tikzpicture}[every node/.style={},xscale=1.5,yscale=1]
 		\foreach \from in {0,...,6}
 		\node at (\from,-0.8){$\from$};
 		\foreach \from in {0,...,6}
 		\node at (-0.6,\from){$\from$};
 		\draw [thick] (-0.3,-0.5) -- (6.6,-0.5);
 		\draw [thick] (-0.3,-0.5) -- (-0.3,6.5);
 		\node (1) at (0,0){$1$};
 		\node (us) at (1,1) {$\overline{\bar u}_{\sigma}^{-1}$};
 		\node (v) at (2,2) {$\overline{\bar u}_{\lambda}^{-1}k^2$};
 		\node (v') at (3,3) {$\overline{\bar u}_{\lambda}^{-1}k^2$};
 		\node at (4,4)   {$\overline{\bar u}_{\sigma}^{-1}\overline{\bar u}_{\sigma}^{-2}$};
 		\node at (5,5) {$\overline{\bar u}_{\sigma}^{-2}\overline{\bar u}_{\lambda}^{-1}$};
 		\node at (6,6) {$\overline{\bar u}_{\sigma}^{-1}\overline{\bar u}_{\lambda}^{-2}k^2$};
 		\draw [->] (v) -- (v');
 		\end{tikzpicture}
 	\end{center}
 \end{figure}
 \thispagestyle{empty}
 \begin{figure}[H]
 	\begin{center}
 		\begin{tikzpicture}[every node/.style={},xscale=1.5,yscale=1]
 		\node at (0,-0.8) {$\sigma-1$};
 		\node at (1,-0.8) {$\sigma$};
 		\foreach \from in {1,...,5}
 		\node at (\from+1,-0.8){$\sigma+\from$};
 		\foreach \from in {0,...,6}
 		\node at (-0.6,\from){$\from$};
 		\draw [thick] (-0.3,-0.5) -- (6.6,-0.5);
 		\draw [thick] (-0.3,-0.5) -- (-0.3,6.5);
 		\node at (0,0){$u_{\sigma}$};
 		\node at (1,0){$a_{\sigma}$};
 		\node (1) at (1,1) {$1$};
 		\node (v) at (2,2) {$v\bar u_{\sigma}$};
 		\node at (3,3) {$v\bar u_{\sigma}$};
 		\node at (4,4)   {$x_{0,1}$};
 		\node at (4,5) {$\frac{x_{0,1}}{a_{\sigma}}$};
 		\node at (5,5) {$\frac{x_{0,1}}{u_{\sigma}}$};
 		\node at (4,6) {$\frac{v}{\bar a_{\lambda}}$};
 		\node at (5,6) {$\bar s$};
 		\node at (6,6) {$\frac{v}{\bar u_{\lambda}}$};
 		\draw [->] (1) -- (v);
 		\end{tikzpicture}
 		\begin{tikzpicture}[every node/.style={},xscale=1.5,yscale=1]
 		\node at (0,-0.8) {$\sigma-1$};
 		\node at (1,-0.8) {$\sigma$};
 		\foreach \from in {1,...,5}
 		\node at (\from+1,-0.8){$\sigma+\from$};
 		\foreach \from in {0,...,6}
 		\node at (-0.6,\from){$\from$};
 		\draw [thick] (-0.3,-0.5) -- (6.6,-0.5);
 		\draw [thick] (-0.3,-0.5) -- (-0.3,6.5);
 		\node at (0,0){$\bar u_{\sigma}$};
 		\node (1) at (1,1) {$1$};
 		\node (v) at (2,2) {$v\bar u_{\sigma}k^2$};
 		\node (v') at (3,3) {$v\bar u_{\sigma}k^2$};
 		\node at (4,4)   {$v$};
 		\node at (5,5) {$v\bar u_{\sigma}^{-1}$};
 		\node at (4,6) {$\frac{v}{\bar a_{\lambda}}k^2$};
 		\node at (5,6) {$\bar sk^2$};
 		\node at (6,6) {$\frac{v}{\bar u_{\lambda}}k^2$};
 		\draw [->] (1) -- (v);
 		\draw [->] (v) -- (v');
 		\end{tikzpicture}
 		\begin{tikzpicture}[every node/.style={},xscale=1.5,yscale=1]
 		\node at (0,-0.8) {$\sigma-1$};
 		\node at (1,-0.8) {$\sigma$};
 		\foreach \from in {1,...,5}
 		\node at (\from+1,-0.8){$\sigma+\from$};
 		\foreach \from in {0,...,6}
 		\node at (-0.6,\from){$\from$};
 		\draw [thick] (-0.3,-0.5) -- (6.6,-0.5);
 		\draw [thick] (-0.3,-0.5) -- (-0.3,6.5);
 		\node at (0,0){$\overline{\bar u}_{\sigma}$};
 		\node (1) at (1,1) {$1$};
 		\node (v) at (2,2) {$\overline{\bar u}_{\sigma}\overline{\bar u}_{\lambda}^{-1}k^2$};
 		\node (v') at (3,3) {$\overline{\bar u}_{\sigma}\overline{\bar u}_{\lambda}^{-1}k^2$};
 		\node at (4,4)   {$\overline{\bar u}_{\lambda}^{-1}$};
 		\node at (5,5) {$\overline{\bar u}_{\sigma}^{-1}\overline{\bar u}_{\lambda}^{-1}$};
 		\node at (6,6) {$\overline{\bar u}_{\lambda}^{-2}k^2$};
 		\draw [->] (v) -- (v');
 		\end{tikzpicture}
 	\end{center}
 \end{figure}
 \thispagestyle{empty}
 \begin{figure}[H]
 	\thispagestyle{empty}
 	\begin{center}
 		\begin{tikzpicture}[every node/.style={},xscale=1.5,yscale=1]
 		\node at (0,-0.8) {$\lambda-2$};
 		\node at (1,-0.8) {$\lambda-1$};
 		\node at (2,-0.8) {$\lambda$};
 		\foreach \from in {1,...,4}
 		\node at (\from+2,-0.8){$\lambda+\from$};
 		\foreach \from in {0,...,6}
 		\node at (-0.6,\from){$\from$};
 		\draw [thick] (-0.3,-0.5) -- (6.6,-0.5);
 		\draw [thick] (-0.3,-0.5) -- (-0.3,6.5);
 		\node at (0,0){$u_{\lambda}$};
 		\node at (1,0){$\frac{a_{\sigma}u_{\lambda}}{u_{\sigma}}$};
 		\node at (2,0){$a_{\lambda}$};
 		\node at (1,1) {$\frac{u_{\lambda}}{u_{\sigma}}$};
 		\node at (2,1) {$\frac{a_{\sigma}u_{\lambda}}{u_{\sigma}^2}$};
 		\node at (2,2) {$1$};
 		\node at (3,3) {$1$};
 		\node at (4,4)   {$0$};
 		\node at (5,5) {$\theta$};
 		\node at (6,6) {$v\bar u_{\sigma}^{-1}$};
 		\end{tikzpicture}
 		\begin{tikzpicture}[every node/.style={},xscale=1.5,yscale=1]
 		\node at (0,-0.8) {$\lambda-2$};
 		\node at (1,-0.8) {$\lambda-1$};
 		\node at (2,-0.8) {$\lambda$};
 		\foreach \from in {1,...,4}
 		\node at (\from+2,-0.8){$\lambda+\from$};
 		\foreach \from in {0,...,6}
 		\node at (-0.6,\from){$\from$};
 		\draw [thick] (-0.3,-0.5) -- (6.6,-0.5);
 		\draw [thick] (-0.3,-0.5) -- (-0.3,6.5);
 		\node at (0,0){$\bar u_{\lambda}$};
 		\node at (1,0){$\sqrt{\bar a_{\lambda}\bar u_{\lambda}}$};
 		\node at (2,0){$\bar a_{\lambda}$};
 		\node at (1,1) {$\bar u_{\lambda}\bar u_{\sigma}^{-1}$};
 		\node at (2,1) {$\sqrt{\bar a_{\lambda}\bar u_{\lambda}}\bar u_{\sigma}^{-1}$};
 		\node at (3,1) {$\bar a_{\lambda}\bar u_{\sigma}^{-1}$};
 		\node (1) at (2,2) {$1k^2$};
 		\node (1') at (3,3) {$1k^2$};
 		\node at (4,4)   {$\bar u_{\sigma}^{-1}$};
 		\node (u2) at (5,5) {$\bar u_{\sigma}^{-2}$};
 		\node (vu) at (6,6) {$v\bar u_{\sigma}^{-1}k^2$};
 		\draw [->] (1) -- (1');
 		\draw [->] (u2) -- (vu);
 		\end{tikzpicture}
 		\begin{tikzpicture}[every node/.style={},xscale=1.5,yscale=1]
 		\node at (0,-0.8) {$\lambda-2$};
 		\node at (1,-0.8) {$\lambda-1$};
 		\node at (2,-0.8) {$\lambda$};
 		\foreach \from in {1,...,4}
 		\node at (\from+2,-0.8){$\lambda+\from$};
 		\foreach \from in {0,...,6}
 		\node at (-0.6,\from){$\from$};
 		\draw [thick] (-0.3,-0.5) -- (6.6,-0.5);
 		\draw [thick] (-0.3,-0.5) -- (-0.3,6.5);
 		\node at (0,0){$\overline{\bar u}_{\lambda}$};
 		\node at (1,1) {$\overline{\bar u}_{\lambda}\overline{\bar u}_{\sigma}^{-1}$};
 		\node (1) at (2,2) {$1k^2$};
 		\node (1') at (3,3) {$1k^2$};
 		\node at (4,4)   {$\overline{\bar u}_{\sigma}^{-1}$};
 		\node (u2) at (5,5) {$\overline{\bar u}_{\sigma}^{-2}$};
 		\node (vu) at (6,6) {$\overline{\bar u}_{\sigma}^{-1}\overline{\bar u}_{\lambda}^{-1}k^2$};
 		\draw [->] (1) -- (1');
 		\end{tikzpicture}
 	\end{center}
 	\thispagestyle{empty}
 \end{figure}
 \thispagestyle{empty}
 \thispagestyle{empty}
 \begin{figure}[H]
 	\thispagestyle{empty}
 	\begin{center}
 		\begin{tikzpicture}[every node/.style={},xscale=1.5,yscale=1]
 		\node at (0,-0.8) {$\rho-4$};
 		\node at (1,-0.8) {$\rho-3$};
 		\node at (2,-0.8) {$\rho-2$};
 		\node at (3,-0.8) {$\rho-1$};
 		\node at (4,-0.8) {$\rho$};
 		\foreach \from in {1,...,2}
 		\node at (\from+4,-0.8){$	\rho+\from$};
 		\foreach \from in {0,...,6}
 		\node at (-0.6,\from){$\from$};
 		\draw [thick] (-0.3,-0.5) -- (6.6,-0.5);
 		\draw [thick] (-0.3,-0.5) -- (-0.3,6.5);
 		\node at (0,0){$u_\sigma u_{\lambda}$};
 		\node at (1,0){$a_{\sigma}u_{\lambda}$};
 		\node at (2,0){$u_\sigma a_{\lambda}$};
 		\node at (3,0){$a_\sigma a_{\lambda}$};
 		\node at (1,1) {$u_{\lambda}$};
 		\node at (2,1) {$\frac{a_{\sigma}u_{\lambda}}{u_{\sigma}}$};
 		\node at (3,1) {$a_\lambda$};
 		\node (1) at (2,2) {$\bar u_\sigma $};
 		\node (1') at (3,3) {$\bar u_\sigma$};
 		\node at (4,4)   {$1$};
 		\node at (5,5) {$0$};
 		\node at (6,6) {$v$};
 		\end{tikzpicture}
 		\begin{tikzpicture}[every node/.style={},xscale=1.5,yscale=1]
 		\node at (0,-0.8) {$\rho-4$};
 		\node at (1,-0.8) {$\rho-3$};
 		\node at (2,-0.8) {$\rho-2$};
 		\node at (3,-0.8) {$\rho-1$};
 		\node at (4,-0.8) {$\rho$};
 		\foreach \from in {1,...,2}
 		\node at (\from+4,-0.8){$	\rho+\from$};
 		\foreach \from in {0,...,6}
 		\node at (-0.6,\from){$\from$};
 		\draw [thick] (-0.3,-0.5) -- (6.6,-0.5);
 		\draw [thick] (-0.3,-0.5) -- (-0.3,6.5);
 		\node at (0,0){$\bar u_\sigma\bar u_{\lambda}$};
 		\node at (1,0){$\bar u_\sigma\sqrt{\bar a_{\lambda}\bar u_{\lambda}}$};
 		\node at (2,0){$\bar u_\sigma\bar a_{\lambda}$};
 		\node at (1,1) {$\bar u_{\lambda}$};
 		\node at (2,1) {$\sqrt{\bar a_{\lambda}\bar u_{\lambda}}$};
 		\node at (3,1) {$\bar a_{\lambda}$};
 		\node (1) at (2,2) {$\bar u_\sigma k^2$};
 		\node (1') at (3,3) {$\bar u_\sigma k^2$};
 		\node at (4,4)   {$1$};
 		\node (u2) at (5,5) {$\bar u_{\sigma}^{-1}$};
 		\node (vu) at (6,6) {$vk^2$};
 		\draw [->] (1) -- (1');
 		\draw [->] (u2) -- (vu);
 		\end{tikzpicture}
 		\begin{tikzpicture}[every node/.style={},xscale=1.5,yscale=1]
 		\node at (0,-0.8) {$\rho-4$};
 		\node at (1,-0.8) {$\rho-3$};
 		\node at (2,-0.8) {$\rho-2$};
 		\node at (3,-0.8) {$\rho-1$};
 		\node at (4,-0.8) {$\rho$};
 		\foreach \from in {1,...,2}
 		\node at (\from+4,-0.8){$	\rho+\from$};
 		\foreach \from in {0,...,6}
 		\node at (-0.6,\from){$\from$};
 		\draw [thick] (-0.3,-0.5) -- (6.6,-0.5);
 		\draw [thick] (-0.3,-0.5) -- (-0.3,6.5);
 		\node at (0,0){$\overline{\bar u}_{\sigma}\overline{\bar u}_{\lambda}$};
 		\node at (1,1) {$\overline{\bar u}_{\lambda}$};
 		\node (1) at (2,2) {$\overline{\bar u}_{\sigma}k^2$};
 		\node (1') at (3,3) {$\overline{\bar u}_{\sigma}k^2$};
 		\node at (4,4)   {$1$};
 		\node (u2) at (5,5) {$\overline{\bar u}_{\sigma}^{-1}$};
 		\node (vu) at (6,6) {$\overline{\bar u}_{\lambda}^{-1}k^2$};
 		\draw [->] (1) -- (1');
 		\end{tikzpicture}
 	\end{center}
 	\thispagestyle{empty}
 \end{figure}
 \thispagestyle{empty}
 \begin{figure}[H]
 	\thispagestyle{empty}
 	\begin{center}
 		\begin{tikzpicture}[every node/.style={},xscale=1.5,yscale=1]
 		\node at (0,-0.8){	 $2\sigma-2$};
 		\node at (1,-0.8){ $2\sigma-1$};
 		\node at (2,-0.8){ $2\sigma$};
 		\foreach \from in {1,...,4}
 		\node at (\from+2,-0.8){ $2\sigma+\from$};
 		\foreach \from in {0,...,6}
 		\node at (-0.6,\from){$\from$};
 		\draw [thick] (-0.3,-0.5) -- (6.6,-0.5);
 		\draw [thick] (-0.3,-0.5) -- (-0.3,6.5);
 		\node at (0,0){$u_{\sigma}^2$};
 		\node at (1,0){$a_\sigma u_\sigma$};
 		\node at (2,0){$a_\sigma^2$};
 		\node (t) at (1,1) {$u_\sigma$};
 		\node (t2) at (2,1) {$a_\sigma$};
 		\node (a) at (2,2) {$v\bar u_\sigma^2$};
 		\node (b) at (3,3) {$v\bar u_\sigma^2$};
 		\node at (4,4) {$0$};
 		\node at (5,5) {$x_{0,1}$};
 		\node at (4,6) {$\frac{v\bar u_\sigma}{\bar a_\lambda}$};
 		\node at (5,6) {$\frac{v\bar u_\sigma}{\sqrt{\bar a_\lambda\bar u_\lambda}}$};
 		\node at (6,6) {$\frac{v\bar u_\sigma}{\bar u_\lambda}$};
 		\draw [->] (t) -- (a);
 		\draw [->,dashed] (t2) -- (b);
 		\end{tikzpicture}
 		\begin{tikzpicture}[every node/.style={},xscale=1.5,yscale=1]
 		\node at (0,-0.8){ $2\sigma-2$};
 		\node at (1,-0.8){$2\sigma-1$};
 		\node at (2,-0.8){$2\sigma$};
 		\foreach \from in {1,...,4}
 		\node at (\from+2,-0.8){$2\sigma+\from$};
 		\foreach \from in {0,...,6}
 		\node at (-0.6,\from){$\from$};
 		\draw [thick] (-0.3,-0.5) -- (6.6,-0.5);
 		\draw [thick] (-0.3,-0.5) -- (-0.3,6.5);
 		\node at (0,0){$\bar u_{\sigma}^2$};
 		\node (t) at (1,1) {$\bar u_\sigma$};
 		\node (a) at (2,2) {$v\bar u_\sigma^2k^2$};
 		\node (a2) at (3,3) {$v\bar u_\sigma^2k^2$};
 		\node at (4,4) {$v\bar u_\sigma$};
 		\node at (5,5) {$v$};
 		\node (b) at (4,6) {$\frac{v\bar u_\sigma}{\bar a_\lambda}k^2$};
 		\node at (5,6) {$\frac{v\bar u_\sigma}{\sqrt{\bar a_\lambda\bar u_\lambda}}k^2$};
 		\node (b2) at (6,6) {$\frac{v\bar u_\sigma}{\bar u_\lambda}k^2$};
 		\draw [->] (a) -- (a2);
 		\draw [->] (t) -- (a);
 		\end{tikzpicture}
 		\begin{tikzpicture}[every node/.style={},xscale=1.5,yscale=1]
 		\node at (0,-0.8){ $2\sigma-2$};
 		\node at (1,-0.8){$2\sigma-1$};
 		\node at (2,-0.8){$2\sigma$};
 		\foreach \from in {1,...,4}
 		\node at (\from+2,-0.8){$2\sigma+\from$};
 		\foreach \from in {0,...,6}
 		\node at (-0.6,\from){$\from$};
 		\draw [thick] (-0.3,-0.5) -- (6.6,-0.5);
 		\draw [thick] (-0.3,-0.5) -- (-0.3,6.5);
 		\node at (0,0){$\overline{\bar u}_{\sigma}^2$};
 		\node at (1,1) {$\overline{\bar u}_\sigma$};
 		\node (a) at (2,2) {$\overline{\bar u}_\sigma^2 \overline{\bar u}_\lambda^{-1} k^2$};
 		\node (b) at (3,3) {$\overline{\bar u}_\sigma^2 \overline{\bar u}_\lambda^{-1} k^2$};
 		\node at (4,4) {$\overline{\bar u}_\sigma\overline{\bar u}_\lambda^{-1}$};
 		\node at (5,5) {$\overline{\bar u}_\lambda^{-1} $};
 		\node at (6,6) {$\overline{\bar u}_\sigma\overline{\bar u}_\lambda^{-2} k^2$};
 		\draw [->] (a) -- (a2);
 		\end{tikzpicture}
 	\end{center}
 	\thispagestyle{empty}
 \end{figure}
 \thispagestyle{empty}
 \thispagestyle{empty}

 \section{\texorpdfstring{The $RO(C_4)$ homology of a point in $\F_2$ coefficients}{The RO(C4) homology of a point in F2 coefficients}}\label{AppendixPoint}
 
 In this appendix, we write down the detailed computation of $k_{\bigstar}$ for $\bigstar\in RO(C_4)$. We use the following notation for Mackey functors (compare with \cite{Geo19}).
 
 \begin{equation}
 k=\begin{tikzcd}
 k\ar[d, "1" left, bend right]\\
 k\ar[u, "0" right,bend right]\ar[d, "1" left, bend right]\\
 k\ar[u, "0" right,bend right]
 \end{tikzcd}\quad \quad \quad
 k_{-}=\begin{tikzcd}
 0\ar[d, bend right]\\
 k\ar[u,bend right]\ar[d, "1" left, bend right]\\
 k\ar[u, "0" right,bend right]
 \end{tikzcd}
\quad \quad \quad
 \ev{k}=\begin{tikzcd}
 k\ar[d, bend right]\\
 0\ar[u,bend right]\ar[d, bend right]\\
 0\ar[u, bend right]
 \end{tikzcd}
 \quad \quad \quad
 \overline{\ev{k}}=\begin{tikzcd}
 0\ar[d, bend right]\\
 k\ar[u,bend right]\ar[d, bend right]\\
 0\ar[u, bend right]
 \end{tikzcd}
 \end{equation}
 \begin{gather}
 L=\begin{tikzcd}
 k\ar[d, "0" left, bend right]\\
 k\ar[u, "1" right,bend right]\ar[d, "0" left, bend right]\\
 k\ar[u, "1" right,bend right]
 \end{tikzcd}\quad \quad \quad
 p^*L=\begin{tikzcd}
 k\ar[d, "0" left, bend right]\\
 k\ar[u, "1" right,bend right]\ar[d, "1" left, bend right]\\
 k\ar[u, "0" right,bend right]
 \end{tikzcd}\quad \quad \quad
 Q=\begin{tikzcd}
k\ar[d, "0" left, bend right]\\
k\ar[u, "1" right,bend right]\ar[d, bend right]\\
0\ar[u,bend right]
\end{tikzcd}
\quad \quad \quad
 Q^{\sharp}=\begin{tikzcd}
k\ar[d, "1" left, bend right]\\
k\ar[u, "0" right, bend right]\ar[d, left, bend right]\\
0\ar[u, right, bend right]
\end{tikzcd}
	\end{gather}
 \begin{equation}
 L^{\sharp}=\begin{tikzcd}
 k\ar[d, "1" left, bend right]\\
 k\ar[u, "0" right,bend right]\ar[d, "0" left, bend right]\\
 k\ar[u, "1" right, bend right]
 \end{tikzcd}\quad \quad \quad
 k_{-}^{\flat}=\begin{tikzcd}
 0\ar[d, bend right]\\
 k\ar[u, bend right]\ar[d, "0" left, bend right]\\
 k\ar[u, "1" right, bend right]
 \end{tikzcd}
 \end{equation} 
 
 Henceforth $n,m\ge 0$. We employ the notation $a|b|c$ to denote the generators of all three levels of a Mackey functor, from top to bottom, used in \cite{Geo19}.
	\subsection{\texorpdfstring{$\boldsymbol{k_*S^{n\sigma+m\lambda}}$}{Sigma plus Lambda oriented}}

	\begin{equation}
	k_*(S^{n\sigma+m\lambda})=
	\begin{cases}\hspace{-0.3em}
	\begin{tabular}{l p{0.5em} l p{1.5em} l}
	$k$			& if &  $*=n+2m$	&&\\
	$Q^{\sharp}$		& if & $n\le *<n+2m$		&and &$*-n$ is even\\
	$Q$		& if & $n+1\le *<n+2m$		&and &$*-n$ is odd\\
	$\ev{k}$& if &  $0\le *<n$&
	\end{tabular}
	\end{cases}
	\end{equation}

	\def\arraystretch{2}
	\setlength{\tabcolsep}{4pt}
	\begin{center}
		\begin{tabular}{p{0.2em} l p{4em} l p{1.5em} l}
			
			$\bullet$& $u_{\sigma}^{n}u_{\lambda}^m|\bar u_{\sigma}^n\bar u_{\lambda}^m|\bar{\bar{u}}_{\sigma}^n \bar{\bar{u}}_{\lambda}^m$		& generates &$k_{n+2m}=k$		&	 &\\

			$\bullet$&$u_{\sigma}^{n}a_{\lambda}^{m-i}u_{\lambda}^{i}|\bar u_{\sigma}^n\bar a_{\lambda}^{m-i}\bar u_{\lambda}^i|\zero$	& generates &$k_{n+2i}=Q^{\sharp}$	&for & $0\le i<m$\\
			
			$\bullet$&$a_{\sigma}u_{\sigma}^{n-1}a_{\lambda}^{m-i}u_{\lambda}^{i}|\bar u_{\sigma}^n\bar a_{\lambda}^{m-i}\bar u_{\lambda}^{i-1}\sqrt{\bar a_{\lambda}\bar u_{\lambda}}|\zero$	& generates &$k_{n+2i-1}=Q$	&for & $1\le i\le m$, $n>0$\\
						
			$\bullet$&$\dfrac{a_{\sigma}a_{\lambda}^{m-i}u_{\lambda}^{i}}{u_{\sigma}}|\bar a_{\lambda}^{m-i}\bar u_{\lambda}^{i-1}\sqrt{\bar a_{\lambda}\bar u_{\lambda}}|\zero$	& generates &$k_{2i-1}=Q$	&for & $1\le i\le m$, $n=0$\\
			
			$\bullet$& $a_{\sigma}^{n-i}u_{\sigma}^{i}a_{\lambda}^m|\zero|\zero$ 	& generates &$k_{i}=\ev{\Z/2}$ 	&for& $0\le i< n$
		\end{tabular}
	\end{center}

	\subsection{\texorpdfstring{$\boldsymbol{k_*S^{-n\sigma-m\lambda}}$}{Minus Sigma minus Lambda oriented}}
	If $n,m$ are not both $0$,
	
	\begin{equation}
	k_*(S^{-n\sigma-m\lambda})=
	\begin{cases}
	\hspace{-0.3em}
	\begin{tabular}{l p{0.5em} l p{1.5em} l p{1.5em} l}
	$L$			& if &  $*=-n-2m$		&and & $m\neq 0$&&\\
	$p^*L$		& if &	$*=-n-2m$		&and & $n>1, m=0$&&\\
	$k_-$		& if &	$*=-1$		&and & $n=1, m=0$&&\\
	$Q^{\sharp}$	& if &	$-n-2m<*<-n-1$	&and & $*+n$ is odd&&\\
	$Q$	& if &	$-n-2m<*<-n-1$	&and & $*+n$ is even&&\\
	$\ev{k}$& if & $-n-1\le *<-1$ &and& $m\neq 0$&&\\
	$\ev{k}$& if & $-n+1\le *<-1$  &and& $m=0$ &&\\
	\end{tabular}
	\end{cases}
	\end{equation}
	
	\def\arraystretch{3}
	\setlength{\tabcolsep}{4pt}
	\begin{center}
		\begin{tabular}{p{0.2em} l p{4em} l p{1.5em} l}
			$\bullet$&	$\Tr_1^4\left(\dfrac1{\bar{\bar{u}}_{\sigma}^{n}\bar{\bar{u}}_{\lambda}^{m}}\right)\Big\rvert \Tr_1^2\left(\dfrac1{\bar{\bar{u}}_{\sigma}^{n}\bar{\bar{u}}_{\lambda}^{m}}\right)\Big\rvert \dfrac1{\bar{\bar{u}} _{\sigma}^{n}\bar{\bar{u}}_{\lambda}^{m}}$		&		 generates & $k_{-n-2m}=L$		 &for &$m\neq 0$\\

			$\bullet$&	$\dfrac{\theta}{u_{\sigma}^{n-2}}\Big\rvert \bar u_{\sigma}^{-n}\Big\rvert\bar{\bar{u}}_{\sigma}^{-n}$  &generates	&$k_{-n}=p^*L$	 &for & $m=0, n\ge 2$ \\
			
			$\bullet$&	$\zero|\bar u_{\sigma}^{-1}|\bar{\bar{u}}_{\sigma}^{-1}$  &generates	&$k_{-1}=k_{-}$	 &for & $n=1$, $m=0$ \\
			
			$\bullet$&	$\dfrac{s}{u_{\sigma}^{n}a_{\lambda}^{i-2}u_{\lambda}^{m-i}}\Big\rvert\dfrac{\bar s}{\bar u_{\sigma}^{n}\bar a_{\lambda}^{i-2}\bar u_{\lambda}^{m-i}}\Big\rvert \zero$ & generates & $k_{-n-2m+2i-3}=Q^{\sharp}$  &for & $2\le i\le m$\\
			
			$\bullet$&$\dfrac{x_{0,1}}{u_{\sigma}^na_{\lambda}^{m-i}u_{\lambda}^{i-1}}\Big\rvert \dfrac{v}{\bar u_{\sigma}^n\bar a_{\lambda}^{m-i}\bar u_{\lambda}^{i-1}} \Big\rvert \zero$& generates & $k_{-n-2i}=Q$  &for & $1\le i< m$\\
						
		$\bullet$&	$\dfrac{x_{0,1}}{a_{\sigma}^{n-i}u_{\sigma}^ia_{\lambda}^{m-1}}\Big\rvert \zero \Big\rvert \zero$ &generates &$k_{-i-2}=\ev{k}$  &for & $0\le i\le n-1$, $m\neq 0$\\
			
			$\bullet$&	$\dfrac{\theta}{a_{\sigma}^{n-i}u_{\sigma}^{n-2}}\Big\rvert \zero \Big\rvert \zero$ &generates &$k_{-i}=\ev{k}$  &for & $2\le i<n$, $m=0$\\
		\end{tabular}
	\end{center}

		\subsection{\texorpdfstring{$\boldsymbol{k_*S^{m\lambda-n\sigma}}$}{Lambda minus Sigma oriented}} If $n,m$ are both nonzero,
	
	\begin{equation}
	k_*(S^{m\lambda-n\sigma})=
	\begin{cases}
	\hspace{-0.3em}
	\begin{tabular}{l p{0.5em} l p{1.5em} l p{1.5em} l}
	$\ev{k}$		& if &$2m-n< *\le -2$ 		&&&&\\
	$k$			& if &  $*=2m-n\ge -1$		& &&&\\
	$\ev{k}\oplus k$			& if &  $*=2m-n\le -2$		& &&&\\
	$Q$	& if &$-1\le *<2m-n$	&and &  $*+n$ is odd &&\\
	$Q^{\sharp}$	& if &$-1\le *<2m-n$	&and &  $*+n$ is even  &&\\
	$\ev{k}	\oplus Q$	& if &$-n+1\le *<2m-n$	&and &  $*+n$ is odd & and & $*\le -2$\\
	$\ev{k}	\oplus Q^{\sharp}$	& if &$-n+2\le *<2m-n$	&and &  $*+n$ is even  & and & $*\le -2$\\
	$Q$			& if & $*=-n$& and &$n\ge 2$&&\\
	$\overline{\ev{k}}$				& if & $*=-1$& and &$n=1$&&\\
	\end{tabular}
	\end{cases}
	\end{equation}

	\begin{center}
		\begin{tabular}{p{0.2em} l p{4em} l p{1.5em} l}
			$\bullet$&	$\dfrac{u_{\lambda}^m}{u_{\sigma}^n}\Big\rvert \dfrac{\bar u_{\lambda}^m}{\bar u_{\sigma}^{n}}\Big\rvert \dfrac{\bar{\bar u}_{\lambda}^m}{\bar{\bar u}_{\sigma}^n}$		&		 generates &  the $k$ in $k_{2m-n}$		 &&\\

			$\bullet$&	$\dfrac{a_{\lambda}^iu_{\lambda}^{m-i}}{u_{\sigma}^{n}}\Big\rvert\dfrac{\bar a_{\lambda}^i\bar u_{\lambda}^{m-i}}{\bar u_{\sigma}^{n}}\Big\rvert \zero$  &generates	& the $Q^{\sharp}$ in $k_{2m-n-2i}$	 &for & $0<i<m$ \\
			
			$\bullet$&	$\dfrac{a_{\sigma}a_{\lambda}^iu_{\lambda}^{m-i}}{u_{\sigma}^{n+1}}\Big\rvert \dfrac{\sqrt{\bar a_{\lambda}\bar u_{\lambda}}\bar a_{\lambda}^i\bar u_{\lambda}^{m-i-1}}{\bar u_{\sigma}^n}\Big\rvert \zero$  &generates	& the $Q$ in $k_{2m-n-2i-1}$	 &for & $0\le i<m$ \\

			$\bullet$&	$\dfrac{\theta a_{\lambda}^m}{a_{\sigma}^{n-i}u_{\sigma}^{i-2}}\Big\rvert \zero\Big\rvert \zero$ & generates & the $\ev{k}$ in $k_{-i}$  &for & $2\le i<n$\\
			
			$\bullet$&	$\dfrac{\theta a_{\lambda}^m}{u_{\sigma}^{n-2}}\Big\rvert \bar a_{\lambda}^m\bar u_{\sigma}^{-n}\Big\rvert \zero$ &generates &$k_{-n}=Q$   & for & $n\ge 2$\\
			
			$\bullet$&	$\zero\Big\rvert \bar u_{\sigma}^{-1}\bar a_{\lambda}^m \Big\rvert \zero$ &generates &$k_{-1}=\overline{\ev{k}}$ & for & $n=1$
		\end{tabular}
	\end{center}

	\subsection{\texorpdfstring{$\boldsymbol{k_*S^{n\sigma-m\lambda}}$}{Sigma minus Lambda oriented}} If $n,m$ are both nonzero,
	
	\begin{equation}
	k_*(S^{n\sigma-m\lambda})=
	\begin{cases}
	\hspace{-0.3em}
	\begin{tabular}{l p{0.5em} l p{1.5em} l p{1.5em} l}
	$Q^{\sharp}$			& if &  $*=n-2$		& and & $n,m\ge 2$&&\\
	$\overline{\ev{k}}$	& if &  $*=-1$		& and & $n=1, m\ge 2$&&\\
	
	$\ev{k}\oplus Q$		& if & $n-2m<*<n-2$	& and & $*+n$ is even& and &$*\ge 0$\\
	$\ev{k}\oplus Q^{\sharp}$		& if & $n-2m<*<n-2$	& and & $*+n$ is odd& and &$*\ge 0$\\
	
	$Q$		& if & $n-2m<*<n-2$	& and & $*+n$ is even& and &$*<0$\\
	$Q^{\sharp}$		& if & $n-2m<*<n-2$	& and & $*+n$ is odd& and &$*<0$\\
		
	$L\oplus \ev{k}$				& if & $*=n-2m$& and &$n-2m\ge 0$&and &  $m\ge 2$\\
	$L$				& if &  $*=n-2m$& and& $n-2m<0$ &and & $m\ge 2$\\
	$L^{\sharp}$				& if &$*=n-2$  &and & $n>1$& and &$m=1$\\
	$k_{-}^{\flat}$				& if &$*=-1$  &and & $n=1$& and &$m=1$\\
	
		$\ev{k}$		& if &$0\le *<n-2m$&&&&	
	\end{tabular}
	\end{cases}
	\end{equation}

	\begin{center}
		\begin{tabular}{p{0.2em} l p{4em} l p{1.5em} l}
			$\bullet$&	$\dfrac{a_{\sigma}^2u_{\sigma}^{n-2}}{a_{\lambda}^{m}}\Big\rvert \dfrac{v\bar u_{\sigma}^n}{\bar a_{\lambda}^{m-1}}\Big\rvert \zero$		&		 generates &  $k_{n-2}=Q^{\sharp}$		 & for & $n,m\ge 2$\\
			
			$\bullet$&	$0\Big\rvert \dfrac{v\bar u_{\sigma}}{\bar a_{\lambda}^{m-1}}\Big\rvert \zero$		&		 generates &  $k_{n-2}=\overline{\ev{k}}$		 & for & $n=1,m\ge 2$\\			
			
			$\bullet$&	$\dfrac{x_{0,2}u_{\sigma}^n}{a_{\lambda}^{i-1}u_{\lambda}^{m-i-1}}\Big\rvert \dfrac{v\bar u_{\sigma}^n}{\bar a_{\lambda}^{i-1}\bar u_{\lambda}^{m-i}}\Big\rvert  0$ &generates &the  $Q$ in $k_{n-2m+2i-2}$  & for & $2\le i\le m-1$\\
						
			$\bullet$&	$ \dfrac{su_{\sigma}^n}{a_{\lambda}^{i-2}u_{\lambda}^{m-i}}\Big\rvert \dfrac{\bar s\bar u_{\sigma}^n}{\bar a_{\lambda}^{i-2}\bar u_{\lambda}^{m-i}}\Big\rvert  0$ &generates &the  $Q^{\sharp}$ in $k_{n-2m+2i-3}$  & for & $2\le i\le m$\\
						
			$\bullet$&	$\dfrac{x_{0,2}u_{\sigma}^n}{u_{\lambda}^{m-2}}\Big\rvert \dfrac{v\bar u_{\sigma}^n}{\bar u_{\lambda}^{m-1}}\Big\rvert  \bar {\bar u}_{\sigma}^n\bar {\bar u}_{\lambda}^{-m}$ &generates &the  $L$ in $k_{n-2m}$  & for & $m\ge 2$\\
			
			$\bullet$&	$\dfrac{a_{\sigma}^2u_{\sigma}^{n-2}}{a_{\lambda}}\Big\rvert v\bar u_{\sigma}^n\Big\rvert \bar{\bar u}_{\sigma}^n\bar{\bar u}_{\lambda}^{-1}$	 &generates &$k_{n-2}=L^{\sharp}$  & for & $n>1, m=1$\\
			
			$\bullet$&	$\zero\Big\rvert v\bar u_{\sigma}\Big \rvert\bar{\bar u}_{\sigma}\bar{\bar u}_{\lambda}^{-1}$ &generates &$k_{-1}=k_{-}^{\flat}$  & for & $n=m=1$\\
			
			$\bullet$&	$\dfrac{a_{\sigma}^{i}u_{\sigma}^{n-i}}{a_{\lambda}^m}\Big\rvert \zero\Big\rvert \zero$ & generates & the $\ev{k}$ in  $k_{n-i}$  &for &  $2< i\le n$\\
		\end{tabular}
	\end{center}

\subsection{Subtleties about quotients} In this subsection, we investigate the subtleties regarding quotients $y/x$, similar to what we did in \cite{Geo19} for the integer coefficient case.

The crux of the matter is as follows: If we have $ax=y$ in $k^{C_4}_{\bigstar}$ then we can immediately conclude that $a=y/x$ as long as $a$ is the \emph{unique} element in its $RO(C_4)$ degree satisfying $ax=y$. Unfortunately, as we can see from the detailed description of $k^{C_4}_{\bigstar}$, there are degrees $\bigstar$ for which $k^{C_4}_{\bigstar}$ is a two dimensional vector space, generated by elements $a,b$ both satisfying $ax=bx=y$; in this case $a,b$ are both candidates for $y/x$ and we need to distinguish them somehow. This is done by looking at the products of $a,b$ with other Euler/orientation classes.

For a concrete example, take $k^{C_4}_{-2+4\sigma-\lambda}$ which is $k^2$ with generators $a,b$ such that $$u_\sigma a=u_\sigma b=\frac{u_\lambda}{u_\sigma^3}$$ 
so both $a,b$ are candidates for $u_\lambda/u_\sigma^4$ (for degree reasons, there is a unique choice for $u_\lambda/u_\sigma^3$). To distinguish $a,b$, we use multiplication by $a_\sigma^2$: for one generator, say $a$, we have $a_\sigma^2a=0$ while for the other generator we get $a_\sigma^2b=\theta a_\lambda$. So now $$a_\sigma^2(a+b)=a_\sigma^2b=\theta a_\lambda$$
and both $a+b,b$ are candidates for $(\theta a_\lambda)/a_\sigma^2$. However, $\theta/a_\sigma^2$ is defined uniquely and we insist
\begin{equation}
	\frac{x}{z}\frac{y}{w}=\frac{xy}{zw}
\end{equation}
whenever $xy\neq 0$, thus $(\theta a_\lambda)/a_\sigma^2$ is uniquely defined from:
\begin{equation}
	\frac{\theta a_\lambda}{a_\sigma^2}=\frac{\theta}{a_\sigma^2}a_\lambda
\end{equation}
Multiplying with $u_\sigma$ returns $0$ and as $u_\sigma b\neq 0$, we conclude that
\begin{equation}
	a+b=\frac{\theta a_\lambda}{a_\sigma^2}
\end{equation}
Since $a_\sigma^2a=0$ and $a_\sigma^2b=\theta a_\lambda$ we conclude that:
\begin{equation}
	a=\frac{u_\lambda}{u_\sigma^4}\text{ , }b=\frac{u_\lambda}{u_\sigma^4}+\frac{\theta a_\lambda}{a_\sigma^2}
\end{equation}

More generally, we can use $u_\sigma$ and $a_\sigma$ multiplication to  distinguish
\begin{gather}
\frac{a_{\lambda}^*u_{\lambda}^{*>0}}{u_{\sigma}^{*}}, \frac{\theta a_{\lambda}^{*>0}}{a_{\sigma}^{*\ge 2}u_{\sigma}^{*}}, \frac{a_{\lambda}^*u_{\lambda}^{*>0}}{u_{\sigma}^{*}}+\frac{\theta a_{\lambda}^{*>0}}{a_{\sigma}^{*\ge 2}u_{\sigma}^{*}}
\end{gather}
Here, $*\ge 0$ is a generic index i.e. the $12$ total instances of $*$ can all be different; the important thing is that the $*$'s are chosen so that these three elements are in the same $RO(C_4)$ degree.\medbreak

We can also distinguish between 
\begin{gather}
\frac{a_{\sigma}a_{\lambda}^*u_{\lambda}^{*>0}}{u_{\sigma}^{*}}, \frac{\theta a_{\lambda}^{*>0}}{a_{\sigma}^{*\ge 2}u_{\sigma}^{*}}, \frac{a_{\sigma}a_{\lambda}^*u_{\lambda}^{*>0}}{u_{\sigma}^{*}}+\frac{\theta a_{\lambda}^{*>0}}{a_{\sigma}^{*\ge 2}u_{\sigma}^{*}}
\end{gather} 
by $u_{\sigma}$ and $a_{\sigma}$ multiplication, although it's easier to use that only the first of the three elements is a transfer.

We distinguish
\begin{gather}
\frac{a_{\sigma}^{*\ge 2}u_{\sigma}^*}{a_{\lambda}^{*}}, \frac{x_{0,2}u_{\sigma}^*}{a_{\lambda}^{*}u_{\lambda}^{*}}, \frac{a_{\sigma}^{*\ge 2}u_{\sigma}^*}{a_{\lambda}^{*}}+\frac{x_{0,2}u_{\sigma}^*}{a_{\lambda}^{*}u_{\lambda}^{*}}
\end{gather} 
by $a_{\lambda}^i$ multiplication  (which for large enough $i$ annihilates only the second term) and $a_{\sigma}$ multiplication (which annihilates only the first term). We similarly distinguish
\begin{gather}
\frac{a_{\sigma}^{*\ge 2}u_{\sigma}^*}{a_{\lambda}^{*}}, \frac{su_{\sigma}^*}{a_{\lambda}^{*}u_{\lambda}^{*}}, \frac{a_{\sigma}^{*\ge 2}u_{\sigma}^*}{a_{\lambda}^{*}}+\frac{su_{\sigma}^*}{a_{\lambda}^{*}u_{\lambda}^{*}}
\end{gather}
by $a_{\lambda}^i$ and $a_{\sigma}^2$ multiplication.

\phantom{1}\smallbreak
\begin{small}
	\noindent  \textsc{Department of Mathematics, University of Chicago}\\
	\textit{E-mail:} \verb|nickg@math.uchicago.edu|\\
	\textit{Website:} \href{http:://math.uchicago.edu/~nickg}{math.uchicago.edu/$\sim$nickg}
\end{small}

\end{document}